\RequirePackage{fix-cm}
\documentclass{svjour3}                     
\smartqed  
\usepackage{graphicx}
\usepackage{amsmath}
\usepackage{amssymb}
\usepackage{bm}
\usepackage{algorithm}
\usepackage{algorithmicx}
\usepackage{url}
\usepackage{graphicx}
\usepackage{algpseudocode}
\usepackage{hyperref}
\usepackage{multirow}
\usepackage{booktabs}

\newcommand{\T}{^\top}
\newcommand{\eps}{\varepsilon}
\newcommand{\G}{\mathcal{G}}

\newcommand{\KK}{\mathcal{K}}

\newcommand{\tx}{\textnormal}

\newcommand{\dist}{\tx{\dist}}

\newcommand{\conv}{\tx{conv}}

\newcommand{\supp}{\text{supp}}

\newcommand{\ad}{\Ab_{\mathcal{G}}}
\newcommand{\lk}{\left\{}
\newcommand{\rk}{\right \} }

\newcommand{\CC}{{\mathcal C}}
\newcommand{\Ab}{{\mathsf A}}

\newcommand{\Eb}{{\mathsf E}}

\newcommand{\Hb}{{\mathsf H}}
\newcommand{\Ib}{{\mathsf I}}

\newcommand{\Ub}{{\mathsf U}}

\newcommand{\Xb}{{\mathsf X}}
\newcommand{\Yb}{{\mathsf Y}}

\newcommand{\UU}{{\mathcal U}}

\renewcommand{\d}{\mathbf d}
\newcommand{\e}{\mathbf e}

\newcommand{\h}{{\mathbf h}}

\newcommand{\oo}{\mathbf o}

\renewcommand{\ss}{\mathbf s}
\renewcommand{\t}{\mathbf t}

\newcommand{\w}{\mathbf w}
\newcommand{\x}{\mathbf x}
\newcommand{\y}{\mathbf y}
\newcommand{\z}{\mathbf z}
\newcommand{\R}{\mathbb R}
\newcommand{\norm}[1]{\|#1\|}

\newcommand{\irg}[2]{[{#1}\!:\!{#2}]}
\newcommand{\che}{{\textstyle{\frac 12}}}

\begin{document}

\title{The Maximum Clique Problem under Adversarial Uncertainty: a  min-max approach
}


\author{
Immanuel Bomze
\and
Chiara Faccio
\and
Francesco Rinaldi
\and
Giovanni Spisso
}

\institute{
Immanuel Bomze \at
Faculty of Mathematics, University of Vienna,\\
Oskar-Morgenstern-Platz~1, 1090 Wien, Austria
\at 
CCOR-CIAS, Corvinus University of Budapest,\\
F\H{o}vám tér~8, 1093 Budapest, Hungary
\at
ILOCOS-L2S, Université Paris-Saclay, CNRS, CentraleSupélec,\\
3, Rue Joliot Curie, 91192 Gif sur Yvette cedex, France\\
\email{immanuel.bomze@univie.ac.at}
\and
Chiara Faccio \and Francesco Rinaldi \and Giovanni Spisso \at
Dipartimento di Matematica ``Tullio Levi-Civita'', Università di Padova,\\
Via Trieste~63, 35131 Padova, Italy\\
\email{\{chiara.faccio,rinaldi\}@unipd.it, giovanni.spisso@phd.unipd.it}
}

\date{Received: date / Accepted: date}

\maketitle

\begin{abstract}
We analyze the problem of identifying large cliques in graphs that are affected by adversarial uncertainty. More specifically, we consider a new formulation, namely the adversarial maximum clique problem, which extends the classical maximum-clique problem to graphs with edges strategically perturbed by an adversary. The proposed mathematical model is thus formulated as a two-player zero-sum game between a clique seeker and an opposing agent. Inspired by regularized continuous reformulations of the maximum-clique problem, we derive a penalized continuous formulation leading to  a nonconvex and nonsmooth optimization problem. We further introduce the notion of stable global solutions, namely points remaining optimal under small perturbations of the penalty parameters, and prove an equivalence between stable global solutions of the continuous reformulation and largest cliques that are common to all the adversarially perturbed graphs. In order to solve the given nonsmooth problem, we develop a first-order and projection-free algorithm based on generalized subdifferential calculus in the sense of Clarke and Goldstein, and establish global sublinear convergence rates   for it. Finally, we report numerical experiments on benchmark instances showing that the proposed method efficiently detects large common cliques.

\end{abstract}

\section{Introduction}
\label{sec:intro}
Denote by $[a\!:\!b]$ the set of all integers in the interval $[a,b]$. Let $\G = (V, E)$  be an undirected, loopless graph on  vertices $V= [1\!:\! n]$ with edge set $E\subset \binom{V}{2}$, and $\ad$ be the adjacency matrix of $\G$, defined as
\begin{equation}\label{adj}
[\ad]_{i,j} = \begin{cases}
1 & \text{ if } \{i,j\} \in E_\G\, ,  \\
0 & \text{ else. }
\end{cases}
\end{equation}
Given any nonempty subset $S$ of $V$, we denote with $\x(S)$ its associated characteristic vector, defined as $\x(S)=  \frac{1}{|S|}\sum\limits_{i\in S} \e_i$, where $\e_i\in \R^n$ is  the $i$-th column of the $n \times n$ identity matrix $\Ib_n$.
A clique in $\G$ is a subset of nodes $C \subseteq V$ that are all mutually connected. It is maximal if there does not exist any larger clique $D
$ such that $C$ is strictly contained in $D$. The classical maximum clique problem consists in finding a maximal clique in $\G$ such that its cardinality (number of vertices) is maximum. Given the graph $\G$, the maximum clique  cardinality is indicated  by $\omega(\G)$. 

The maximum clique problem has practical applications in a wide range of fields including social network analysis, scheduling, financial networks and telecommunications. Therefore, despite being an NP-hard problem \cite{Karp1972}, several solution methods have been explored (see, e.g., \cite{bomze1999maximum,wu2015review}) due to its practical significance.

In \cite{motzkin-straus-1965}, Motzkin and Straus provided a connection between the maximum clique problem and the solutions of the nonconvex quadratic  (so,  continuous) optimization problem 
\begin{equation}\label{MotzkinStraus_formulation}
z^* := \max_{\x \in \Delta^V} \x\T \ad  \x
\end{equation}
where $\Delta^V=\{\x\in \R^n_+: \e\T \x=1  \}$ and $\e$ is the all-ones vector. In particular, they proved that given $C \subseteq V$ a maximum  clique in $\G$, the characteristic vector $\x(C)$ is a global  maximizer of (\ref{MotzkinStraus_formulation}), with the clique number $\omega(\G)$ encoded in its optimal value: $z^* = 1- \frac 1{\omega(\G)}$. A drawback linked to the Motzkin-Straus formulation is the existence of spurious solutions \cite{pelillo1995relaxation,pelillo1995feasible}, i.e., local (global) solutions of (\ref{MotzkinStraus_formulation}) that are not characteristic vectors and have no connection to a  clique in the graph.

In order to deal with this issue, a regularized approach was hence proposed by Bomze in \cite{Bomz97b}:
\begin{equation}
\label{msreg}
\max_{\x \in \Delta^V} \x\T [\che\Ib_n +\ad ] \x
\end{equation}
where $\Ib_n$ denotes the $n\times n$ identity matrix. This encoding guarantees that any  local (global) solution $\x^*$ to problem (\ref{msreg}) corresponds to a maximal (maximum) clique $C = \supp(\x^*) = \{i \in V: x_i >0 \}$ and vice versa. This one-to-one correspondence is ensured by the regularization term ${\norm\x}_2^2= \x\T \Ib_n\x$.

The role of the regularization term was further generalized in \cite{hungerford2019general}, where the authors analyzed different variants of the regularized continuous formulation (\ref{msreg}), and established conditions that guarantee the equivalence between the regularized formulation and the maximum clique problem, both in terms of local and global solutions.
In \cite{Stozhkov2022Continuous}, the Motzkin-Straus formulation was extended to two clique relaxation models, $s$-defective clique and $s$-plex, in terms of maximizing a continuous cubic function over a polyhedral set described by two variables $\x$ and $\y$. In \cite{Bomz21}, the authors provided a regularized version of the cubic continuous formulation for the maximum $s$-defective clique problem and applied some tailored variants of the Frank-Wolfe algorithm to their formulation.

In scenarios where the exact structure of $\G$ cannot be precisely determined, the problem data, that is $\ad$,  can be considered  given subject to some uncertainty. To model this, one viable approach is to account for continuous perturbations in the entries of $\ad$ within a predefined uncertainty set (e.g., a ball, an ellipsoid, a spectrahedron, or a box).
This model makes sense if the link strengths on edges vary, as explored in the literature, e.g., 
by~\cite{Bomz20a},  where the authors focus on relevant applications in social network analysis.

In this paper, we consider an unweighted graph $\G$ and analyze the scenario where the perturbation is discrete, meaning that some edges may be added to or removed from $\G$, i.e., some entries of $\ad$ may switch from zero to one. In many modern applications, such uncertainty is adversarial in nature, as different agents have conflicting goals regarding the detection or obscuration of the underlying graph structure.  This dynamics can be hence naturally modeled as a two-player game between a clique seeker and a network adversary. Formally, we have a family of possible graph structures\ $\G_1,\dots, \G_m $, each representing a possible configuration of edges that may occur due to  adversarial intervention.
 The seeker’s goal is thus to maximize the connectivity of the selected nodes, while the adversary chooses a graph configuration $\G_i$ that minimizes this objective by removing/masking critical edges or by including fake ones. We thus define the \emph{backbone network}, as the set of edges that are stable, i.e., included in all the graph configurations and the \emph{adversarial maximum common clique problem} as the seeker's task of finding the largest subset of vertices forming a clique in the backbone network (i.e., a clique with edges in every configuration), without explicitly knowing all the graph structures. We hence have an adversarial interaction between the two players, and, whenever the seeker selects a subset of vertices, the adversary responds by revealing the worst possible graph configuration~$\G_i$ among the available ones, that is, the one in which the induced subgraph on the selected subset of vertices is least connected. This problem formulation may appear in several real-world settings:
\begin{itemize}
\item \textbf{Communication Networks~\cite{boutremans2002impact,Kleinberg2008Network}:} The backbone network consists of key routers that have stable connections, while additional links are available depending on congestion or failures. Users hence attempt to identify the most connected subset of nodes, while the network owner dynamically modifies link availability to obscure this information.
\item \textbf{Cybersecurity and Privacy \cite{srivastava2018graph}:} A hidden clique represents a group of critical servers or key actors in a communication network.  Attackers attempt to reveal the structure by probing connections, while defenders limit information leakage through controlled link availability.
\item \textbf{Social Network Analysis \cite{waniek2018hiding,zhou2008brief,zhu2024defending}:}
 Influential subgroups in a social network form hidden cliques, and malicious users try to identify these communities to feed them with fake news and rumors to manipulate public opinion or generate revenue on their sites (e.g., clickbait). The platform provider may hence intentionally obfuscate certain relationships to prevent manipulation or unwanted discoveries.
\end{itemize}

As we will see in the next subsection, a natural way to extend the Motzkin – Straus formulation to this case would lead to an adversarial max-min optimization problem. However, unlike the classical maximum-clique case, the solutions of this max-min problem do not, in general, correspond to characteristic vectors of common cliques and even non-feasible patterns may appear. Thus, in order to guarantee the discrete-continuous equivalence, we need to introduce a suitable continuous reformulation that generalizes the above described  regularization to the adversarial case.

\subsection{Preliminary formulation }
\label{sec:subsec:preliminary}
We denote the finite universe of all regularized adjacency matrices of graphs of order $n$ by
$$\overline \UU :=\{ \che \Ib_n + \ad : \G \mbox{ is a graph of order }n\}\, ;$$
In other words, $\overline \UU$ is composed of all $\Ub=\Ub\T$ such that $\Ub - \frac 12 \Ib_n\in   \{0,1\}^{n\times n}$ is a symmetric binary matrix with zero diagonal. Reverting the dependence on $\G$, we may write $\G 
= \G_{\Ub} =(V,E_\Ub)$ if and only if $\Ub = \che\Ib_n + \ad$. 

For the theoretical analysis to follow, we may hence consider a general (finite) uncertainty set $\UU \subseteq \overline\UU $. We consider the following robust optimization problem:

\begin{equation}\label{grob}
\max_{\x \in \Delta^V} \min_{\Ub \in \mathcal{U} }\x\T  \Ub  \x\, ,
\end{equation}
a continuous, generally nonsmooth, nonconvex, piecewise quadratic optimization problem over the simplex polytope. Finiteness of $\UU$ and compactness of $\Delta^V$ ensure the existence of all extrema considered. It is important to highlight that the seeker does not have direct access to the set of all configurations $\mathcal{U}=\{\Ub_1,\ \dots,\ \Ub_m\}$. If this were the case, the seeker might easily calculate the backbone network adjacency matrix  $\Ab_{\G_{bb}}$ (considering the Hadamard product of all the available adjacency matrices) and then solve the classic max-clique problem given in~\eqref{msreg}. 

Since $\Ub$ will be indefinite for all $\Ub\in \UU$, Sion's theorem cannot be applied, and there is only a general min-max inequality. Indeed, the min-max problem
$$\min _{\Ub\in \UU} \max_{\x\in \Delta^V} \x\T\Ub \x$$
corresponds to selecting the lowest $\omega(\G_\Ub)$ across all graphs $\G_\Ub$ generated by $\Ub \in \UU$, which is obviously a different problem. 
The solution to this problem may be complicated by the fact that there is more than one minimal $\Ub\in \UU$, but if solved, it amounts to selecting those edges for removal such that the resulting graph has the smallest possible clique number.

As an illustrative case, let us consider the following example. Let $\G = \KK_4$ be a complete graph with four nodes and choose $\UU = \{\Ub_1, \Ub_2, \Ub_3, \Ub_4\}$ as shown in Figure \ref{example_k4}. It is easy to prove that  $\min_{\Ub \in \UU } \max_{\x \in \Delta^V} \x\T  \Ub  \x = \frac 34 \, .$
    \begin{figure}[ht!]
    \centering    \includegraphics[width=0.3\textwidth]{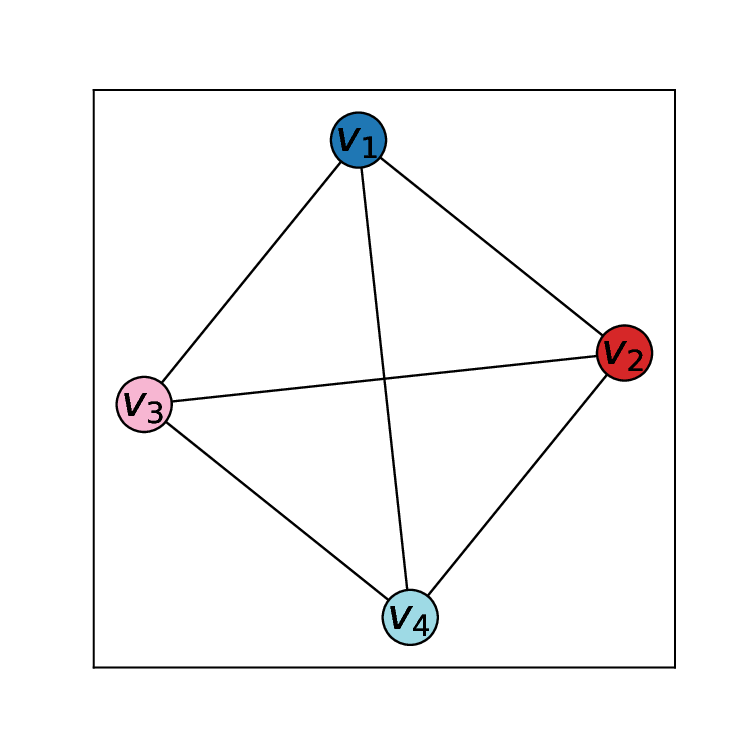}
\includegraphics[width=0.3\textwidth]{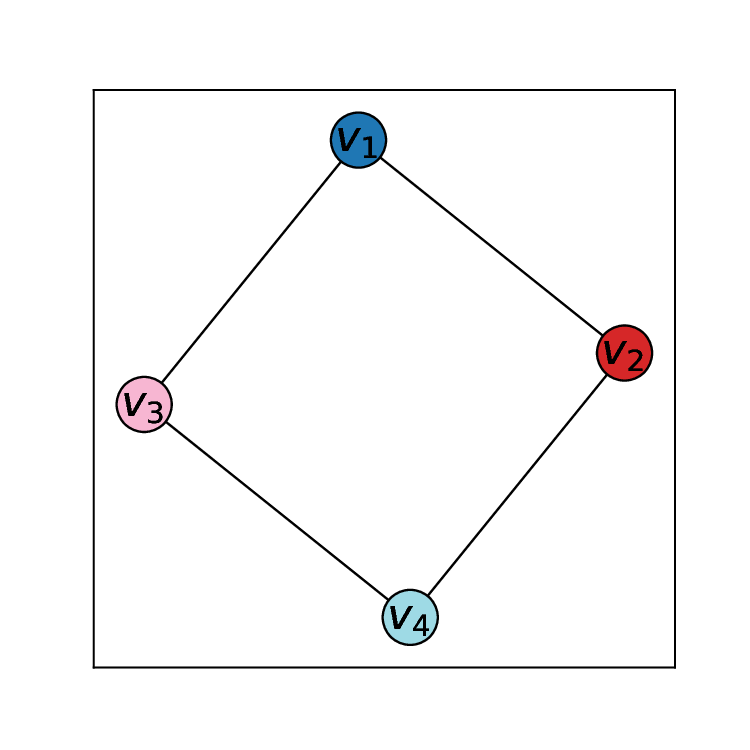}    \includegraphics[width=0.3\textwidth]{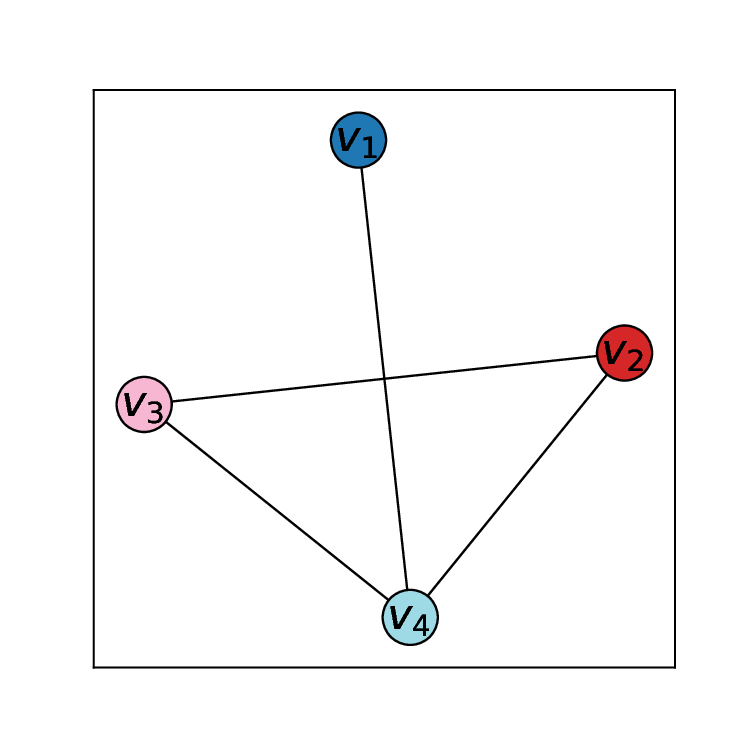}
\includegraphics[width=0.3\textwidth]{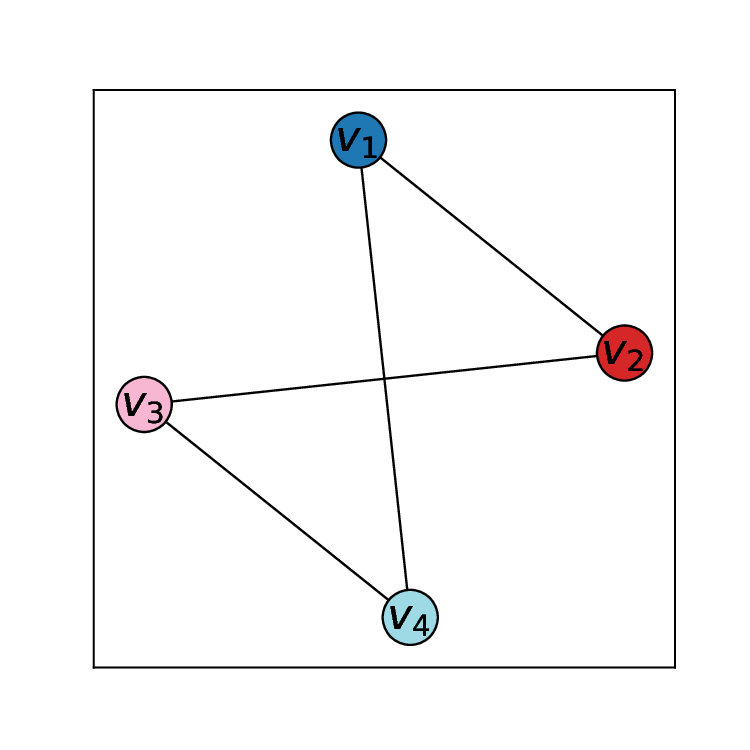}    \includegraphics[width=0.3\textwidth]{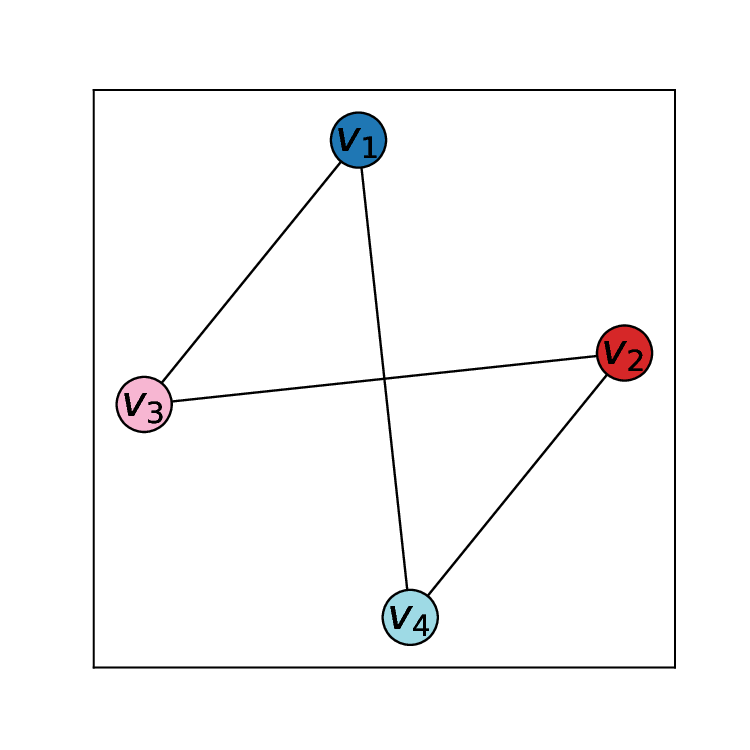}
    \caption{original graph $\G$, and edge configurations $\G(\Ub_1)$, $\G(\Ub_2)$, $\G(\Ub_3)$, $\G(\Ub_4)$  that may occur due to  adversarial intervention.
    }
    \label{example_k4}
\end{figure}
In this example, the largest common cliques are the single vertices and, in this case, $\x\T \Ub_ i \x  =\che$ for all $i\in \irg 14$ and for each $\x=\e_i$, the characteristic vector associated with a single vertex $i\in V$. However, if $\x\T=\frac 14[1,1,1,1]$, we have $  \x\T \Ub_i \x =  \frac 58 > \che$ for all $i \in \irg 14$. Hence, in general, the characteristic vector $\frac{1}{|C|}\sum\limits_{i\in C} \e_i$
of the largest common clique $C$ is not an optimal solution to the problem~\eqref{grob}.
 
In the previous example, we observed that for an optimal solution  $\x$ to the problem~$\eqref{grob}$, its support $\supp(\x)$ need not be a common clique for all the modifications of $\G$. This motivates the following strategy: adding (in fact, subtracting, as we are maximizing in $\x$) a term penalizing deviations from the complete graph $\KK_n$. This will enable a complete characterization of (even) global solutions in terms of maximum common cliques across $\UU$. Our proof strategy for obtaining this result is inspired by~\cite{hungerford2019general}. However, major changes are necessary due to nonsmoothness.

\subsection{Contributions}
\label{subsec:contributions}
The contributions of this paper are therefore twofold:

\begin{itemize}

\item First, we introduce a novel penalized game-theoretical program that extends the regularized continuous maximum-clique reformulation to adversarial settings. This nonsmooth and nonconvex mathematical model introduces a penal\-ty-based mechanism and a semicontinuous relaxation that together enforce the combinatorial structure of the maximum common clique across all available graph realizations.
Within this framework, we further define the new concept of \emph{stable global solution}, i.e., a solution that persists when the parameters of the penalized model are slightly perturbed.
We establish a rigorous equivalence result showing that stable global solutions of the penalized continuous program correspond exactly to the maximum common cliques in the discrete model and vice versa. This result provides the first exact continuous characterization of the maximum clique problem under adversarial uncertainty, thus extending the meanwhile classical correspondence to an adversarial context.

\item  Second, we develop a first-order and projection-free algorithm tailored to the nonsmooth and nonconvex structure of the proposed formulation, which is particularly well suited to the underlying adversarial setting \footnote{as we will see, the seeker, in the considered framework, interacts with the adversary only through an oracle that returns information related to suitably defined worst-case graph realizations for a given node selection.}. Furthermore, this kind of methods can be easily embedded into a global optimization algorithmic framework based on multistart or local search/basin hopping strategies (see, e.g., \cite{grosso2007population,leary2000global,locatelli2013global} for further details on these classes of algorithms). Since the proposed formulation is nonsmooth, the general theory of differentiation for Lipschitz-continuous functions plays a crucial role in the algorithmic development. Contributions in this area were made by Clarke and Goldstein, as presented in \cite{bagirov2014introduction}. In particular, Goldstein’s foundational work \cite{Goldstein1977} on nonsmooth unconstrained optimization problems inspired some of the theoretical tools needed to analyze our approach. Thanks to those tools, we thus establish  global convergence sublinear rates for the proposed method. Specifically, we demonstrate that employing a Goldstein-like subdifferential with fixed accuracy leads to an approximate Goldstein stationary point, whereas increasing accuracy guarantees convergence to a Clarke stationary point. Extensive numerical experiments further demonstrate that our projection-free algorithm, even when embedded into a quite basic multistart framework, consistently identifies large common cliques of adversarially perturbed networks. 
\end{itemize}

The paper is organized as follows. In Section~\ref{section:mcc}, we develop a continuous, nonsmooth formulation of the maximum common clique problem, and we provide an equivalence between the stable global minimizers of the proposed formulation and the maximum common cliques. In Section~\ref{section:alg},  a new projection-free algorithm is introduced to compute local solutions to the nonsmooth formulation. Section~\ref{section:experiment} reports some experimental results, and in Section~\ref{section:conclusion}, we delineate the conclusions of the paper.

An implementation of the code presented in this paper, along with the tested instances, can be found at the following link: \href{https://github.com/GiovanniSpisso/MCP_under_Adversarial_Uncertainty}{github.com/Maximum\_Clique\_Problem}.

\section{Exact penalization leads to the maximum common clique}
\label{section:mcc}

\subsection{Adding two penalty terms; epigraphic QCQP formulation}
\label{subsec:penalty}

Denote by $\Eb := \e\e\T$ the $n\times n$ all-ones matrix. Then $\KK_n$ has adjacency matrix $\Ab_{\KK_n} = \Eb-\Ib_n$ and is generated by the densest $\overline\Eb:= \Eb-\che \Ib_n= \Ab_{\KK_n}+ \che \Ib_n \in \overline\UU$, in that $\KK_n=\G(\overline \Eb)$. Now for a positive parameter $\beta>0$, we consider the quadratic form
\begin{align}\label{deffU}
& \x\T \Ub \x - \beta  \x\T (\overline \Eb - \Ub) \x \\
& = \x\T \biggl [(1+\beta)\Ub - \beta  \overline \Eb \biggl ]\x  \, ,
\end{align}
an affine combination of the quadratic forms generated by realized $\Ub$ and the densest $\overline \Eb$. Rewritten, we have
\begin{equation}\label{add_penalty}
 \x\T  \Ub \x - \beta \sum_{\substack{i, j \\ i \ne j}} (1 -  \Ub_{i j})x_i x_j \, .
\end{equation}
We also add another set of variables, another penalty term  and a new set of constraints:
\begin{equation}\label{add_second_penalty}
f_\Ub (\x, \y)  :=   \x\T  \Ub \x - \beta \x\T (\overline \Eb - \Ub) \x  + \frac{1}{2 \gamma} \norm{\y}^2\, .
\end{equation}
such that \begin{align*}
    &  \varepsilon y_i \leq x_i \leq y_i \, , \, i \in V\, ,  \\
    & y_i \in \{0,1\}\, , \, i \in V\, ,
\end{align*}
where $\gamma>0$, $0<\varepsilon \leq 1/n$ and $n$ is the order of the graph $\G$.
We observe that, in this way, \begin{align*}
    & y_i = 0 \iff x_i = 0\,; \\
    & y_i = 1 \iff x_i \geq \varepsilon \,.
\end{align*}
In other words, we look at the $x_i$ as semicontinuous variables~\cite{Sun2013}.
We define
 $$F(\x, \y) :=  \min_{\Ub \in \UU} f_\Ub (\x, \y)\, ,$$
 to introduce the penalized problem
\begin{equation}\label{original_new_var}
    \begin{split}
        \max & \quad F(\x, \y) \\
        & \text{ s.t. } \x \in \Delta^V \\
        & \qquad  \varepsilon y_i \leq x_i \leq y_i \, , \, i \in V\, ,  \\
    & \qquad y_i \in \{0,1\}\, , \, i \in V\,.
    \end{split}
\end{equation}
By relaxing the integrality in the previous problem, we have 
\begin{equation}\label{original_new_var_relax}
    \begin{split}
        \max & \quad F(\x, \y) \\
        & \text{ s.t. } \x \in \Delta^V \\
        & \qquad  \varepsilon y_i \leq x_i \leq y_i \, , \, i \in V\, ,   \\
    & \qquad 0 \leq y_i \leq 1\, , \, i \in V   \,.
    \end{split}
\end{equation}
We observe that, in this case, \begin{align*}
    &  y_i = 0 \iff x_i = 0\,; \\
    &  y_i = 1 \implies x_i \geq \varepsilon \,.
\end{align*}
So we change the original feasible set $\Delta^V$ into the set
$$
\Delta^V_{\varepsilon} := \{
(\x, \y)\in \R^{2n} : \x\in \Delta^V ,
\begin{aligned}[t]
& \varepsilon y_i \le x_i \le y_i \mbox{ for all } i \in V\\
& \mbox{and } 0 \le y_i \le 1 \mbox{ for all } i \in V \}\, .
\end{aligned}
$$

We also rewrite the original problem in the following way:
$$\max_{(\x,\y)\in \Delta^V_{\varepsilon}} \min_{\Ub \in \UU} f_\Ub(\x, \y)  = - \min_{(\x,\y)\in \Delta^V_{\varepsilon}} \max_{\Ub \in \UU} g_\Ub(\x, \y) $$
with  $g_\Ub(\x, \y) := - f_\Ub(\x, \y)$. Denoting by $G(\x, \y) = \max\limits_{\Ub \in \UU} g_\Ub(\x, \y)$,  problem~\eqref{original_new_var_relax} is rewritten as  (the negative of) a potentially nonsmooth, nonconvex, constrained min-max problem
\begin{equation}\label{original-v2}
    \begin{split}
        -\min & \quad G(\x, \y) \\
         \text{ s.t. } &\quad (\x,\y)\in \Delta^V_{\varepsilon}\,.
    \end{split}
\end{equation}
Now all component functions $g_\Ub$ of $G$ are quadratic forms separable in $\x$ and $\y$, 
with $$\nabla_{\x} g_\Ub(\x, \y) = \Hb_\Ub \x \,,$$
where $\Hb_\Ub = \nabla^2_{\x \x} g_\Ub = 2\biggl [-\Ub  + \beta (\overline \Eb - \Ub) \biggl ]$ and $$\nabla_{\y} g_\Ub(\x, \y) = - \frac{1}{\gamma}  \y\,.$$
 It follows for any $(\x,\y)\in \Delta^V_{\varepsilon}$ by definition of $\overline \Eb =\Eb-\che\Ib_n$ that the gradients w.r.t.~$\x$ equal
\begin{equation}\label{Hux}
\nabla_{\x}  g_\Ub(\x, \y) = \Hb_\Ub \x = \beta (2\e-\x) - 2(1+\beta)\Ub\x\, ,
\end{equation}
an identity we will use later. For instance, we get
$$(\e_h-\e_k)\T \nabla_{\x}  g_\Ub(\x, \y) = \beta (x_k-x_h) - 2(1+\beta)\sum_j [U_{hj}-U_{kj}]x_j\, .$$

Next consider a smooth epigraphic QCQP reformulation of the negative of problem~\eqref{original-v2}, namely of the nonsmooth problem
\begin{equation}\label{neg-formulation}
\min_{(\x,\y)\in \Delta^V_{\varepsilon}} \max_{\Ub\in \UU}
    g_\Ub(\x, \y) \, .
\end{equation}
This epigraphic equivalent of problem~\eqref{neg-formulation} has one more variable, a linear objective and (nonconvex) quadratic, as well as linear, constraints:
\begin{equation}\label{kkt-formulation}
\min_{(\x,\y,\nu)\in \Delta^V_{\varepsilon}\times\R} \lk  \nu :
    g_\Ub(\x,\y) \leq \nu \mbox{ for all } \Ub\in \UU\rk\, .
\end{equation}

Next we will describe and simplify the KKT (first-order) conditions related to optimality of a point $(\x,\y,\nu)$ for problem~\eqref{kkt-formulation}.

\begin{lemma}\label{kktlem}
The Karush/Kuhn/Tucker (KKT) conditions for problem~\eqref{kkt-formulation} are
\begin{itemize}
    \item stationarity conditions \begin{equation}\label{opt}
    \begin{split}
    & \nabla_{\nu}L(\x, \y, \nu; \bm{\lambda}, \bm{\mu}, \bm{\psi}, \bm{\theta}, \bm{\phi}, \bm{\delta}, \sigma) = 1 - \sum\limits_{\Ub\in \UU} \lambda_\Ub = 0\, , \\
    &  \nabla_{\x}L(\x, \y, \nu; \bm{\lambda}, \bm{\mu}, \bm{\psi}, \bm{\theta}, \bm{\phi}, \bm{\delta}, \sigma) = \sum\limits_{\Ub\in \UU} \lambda_\Ub \nabla_{\x} g_\Ub(\x,\y) - \bm{\mu} - \phi+ \delta - \sigma \e = 0\, , \\
    &  \nabla_{\y}L(\x, \y, \nu; \bm{\lambda}, \bm{\mu}, \bm{\psi}, \bm{\theta}, \bm{\phi}, \bm{\delta}, \sigma) = \sum\limits_{\Ub\in \UU} \lambda_\Ub \nabla_{\y} g_\Ub(\x, \y) - \bm{\psi} + \theta + \varepsilon \phi - \delta = 0\, , 
    \end{split}
    \end{equation}
    \item primal and dual feasibility conditions \begin{equation}\label{feas}
    \begin{split}
    & g_\Ub(\x, \y) \leq \nu \quad \mbox{ for all }\Ub\in \UU\, , \\
    & \x\in \Delta^V\, , \\
    & \varepsilon \y \leq \x \leq \y\,, \\
    & 0 \leq \y \leq \e\,,
   \\
    & \lambda_\Ub \geq 0 \quad \mbox{ for all } \Ub\in \UU\, , \\
    & \mu_i. \phi_i, \delta_i, \psi_i, \theta_i \geq 0 \quad \mbox{ for all }i \in V\,, \end{split}
    \end{equation}
    \item complementarity conditions \begin{equation} \label{compl}
    \begin{split}
    & \lambda_\Ub (g_\Ub(\x, \y) - \nu) = 0 \quad \mbox{ for all } \Ub\in \UU\, ,\\
    & \bm{\mu}\T \x = 0\,, \\
    & \bm{\psi}\T \y = 0\,, \\
    & \bm{\theta}\T (\y - \e) = 0\,, \\
    & \bm{\phi}\T (\x - \varepsilon \y) = 0\,, \\
    & \bm{\delta}\T (\x - \y) = 0\,, \\
    \end{split}
\end{equation}
\end{itemize}
Therefore, under Equations~\eqref{opt} to \eqref{compl} we have
$$\bm{\lambda}\in \Delta^\UU := \{\bm{\lambda}\in \R^n_+: \sum_{\Ub \in \UU} \lambda_{\Ub} = 1  \},$$ for the $\bm{\lambda}$-weighted average $\h^x(\x|\bm{\lambda})$ of the gradients $\nabla_{\x} g_\Ub(\x, \y)$
$$\h^x(\x|\bm{\lambda}):= \sum\limits_{\Ub\in \UU}\lambda_\Ub \nabla_{\x} g_\Ub(\x, \y)\in \R^n $$
it holds $\h^x(\x|\bm{\lambda}) - \sigma \e + \delta =\bm{\mu} + \bm{\phi}\in \R^n_+$ and for the $\bm{\lambda}$-weighted average $\h^y(\y|\bm{\lambda})$ of the gradients $\nabla_{\y} g_\Ub(\x, \y)$
$$\h^y(\y|\bm{\lambda}):= \sum\limits_{\Ub\in \UU}\lambda_\Ub \nabla_{\y} g_\Ub(\x, \y)\in \R^n $$
it holds $\h^y(\y|\bm{\lambda}) + \theta + \varepsilon \phi =\bm{\psi} + \bm{\delta} \in \R^n_+$. In particular, $$\sigma=\x\T\h^x(\x|\bm{\lambda})+ \y\T\h^y(\y|\bm{\lambda}) + \e\T \bm{\theta} = 2\nu + \e\T \bm{\theta}  = 2 G(\x, \y) + \e\T \bm{\theta} \, .$$ 

\end{lemma}

\begin{proof}
The Lagrangian function of the problem~\eqref{kkt-formulation} is
\begin{align*}
L&(\x, \y, \nu; \bm{\lambda}, \bm{\mu}, \bm{\psi}, \bm{\theta},
   \bm{\phi}, \bm{\delta}, \sigma) \\
&= \nu + \sum_{\Ub\in \UU} \lambda_\Ub [g_\Ub(\x, \y) - \nu]
   - \bm{\mu}\T \x - \bm{\psi}\T \y + \bm{\theta}\T (\y - \e) \\
&\quad - \bm{\phi}\T (\x - \varepsilon \y)
   + \bm{\delta}\T (\x - \y)
   - \sigma (\e\T \x - 1)\, .
\end{align*}
Therefore the KKT conditions have the form specified. Moreover, the first condition in~\eqref{opt} and dual feasibility of $\bm{\lambda}$ together give $\bm{\lambda}\in \Delta^\UU$.
Further, $\h^x(\x|\bm{\lambda}) - \sigma \e + \delta =\bm{\mu} + \bm{\phi}$ and $\h^y(\y|\bm{\lambda}) + \theta + \varepsilon \phi =\bm{\psi} + \bm{\delta}$ are exactly the conditions on $\bm{\lambda}$ in~\eqref{opt}, which in turn implies via $\x\in \Delta^V$ and~\eqref{compl} 
\begin{equation}\label{avgr}
\begin{split}
\sigma & = \sigma\, \x\T\e = \x\T\h^x(\x|\bm{\lambda}) +\x\T \bm{\delta} - \x\T\bm{\mu}- \x\T\bm{\phi}  \\
& = \x\T\h^x(\x|\bm{\lambda})+\x\T \bm{\delta} - \x\T\bm{\phi} \\
& = \x\T\h^x(\x|\bm{\lambda})+\y\T \bm{\delta} - \varepsilon \y\T\bm{\phi}  \\
& = \x\T\h^x(\x|\bm{\lambda})+ \y\T\h^y(\y|\bm{\lambda})  + \y\T \bm{\theta}-  \y\T \bm{\psi} \\
& =  \x\T\h^x(\x|\bm{\lambda})+ \y\T\h^y(\y|\bm{\lambda})  + \e\T \bm{\theta} \, .
\end{split}
\end{equation}
Moreover, since $g_\Ub$ are quadratic forms (homogeneous of degree 2), we have by Euler's theorem $\x\T\nabla_{\x} g_\Ub (\x, \y) + \y\T\nabla_{\y} g_\Ub (\x, \y) = 2 g_\Ub(\x, \y)$, which in conjunction with the other complementarity constraint of~\eqref{compl}, and the fact that $\bm{\lambda}\in \Delta^\UU$, entails
\begin{align*}
    \nu & =  \sum\limits_{\Ub\in \UU} \lambda_\Ub g_\Ub(\x, \y) = \che \sum\limits_{\Ub\in \UU} \lambda_\Ub \x\T\nabla_{\x} g_\Ub(\x, \y) + \che \sum\limits_{\Ub\in \UU} \lambda_\Ub \y\T\nabla_{\y} g_\Ub(\x, \y) \\
    & = \che [\x\T \h^x(\x|\bm{\lambda}) +  \y\T \h^y(\y|\bm{\lambda})]= \che (\sigma - \e\T \bm{\theta}) \, ,
\end{align*}
the last equation coming from~\eqref{avgr}. Finally, feasibility of $(\x, \y,\nu)$ means $G(\x, \y)\le\nu$ but on the other hand, for all $\lambda_\Ub>0$ we have $g_\Ub (\x, \y)=\nu$ by~\eqref{compl},
and at least one such $\lambda_\Ub>0$  exists since $\bm{\lambda}\in \Delta^\UU$. So $G(\x, \y) \ge g_\Ub(\x, \y)=\nu$, which establishes $G(\x, \y)=\nu$ under the KKT conditions.

\end{proof}

\subsection{The role of common cliques and their characteristic vectors}
\label{subsec:common_cliques}

Let $S \subseteq V$ be a nonempty subset of nodes, and let $\Delta^V_{\varepsilon}(S)$  defined by
    $$\Delta^V_{\varepsilon}(S) := \{(\x, \y) \in \Delta^V_{\varepsilon}: \supp(\x) = \supp(\y) \subseteq S \}\, .$$
    We consider the system of all common cliques across $\UU$,
    $$\CC_\UU :=\lk C\subseteq V: C\mbox{ is a clique in }\G(\Ub) \mbox{ for all }\Ub\in \UU\rk\, ,$$
    and put

\begin{equation}\label{delta0}
\begin{split}
\Delta^0_{\varepsilon} & := \bigcup_{C \in \CC_\UU}\Delta^V_{\varepsilon}(C) \\
& = \{ (\x, \y) \in \Delta^V_{\varepsilon} : \supp(\x) = \supp(\y) \text{ is a common clique across } \UU \}\, .
\end{split}
\end{equation}

As usual, we will call an element $C\in\CC_\UU$ a {\em maximal common clique} if it is maximal w.r.t.~set inclusion, i.e., there is no other $\tilde C\in \CC_\UU$ containing $C$ as a subset.

\begin{proposition}\label{foco}
Let $\x(C)$ be a characteristic vector of a maximal common clique $C\in \CC_\UU$ and define $\nu:=G(\x(C), \y)$, where $(\x(C), \y) \in \Delta^V_{\varepsilon}$ and $\y \in \{0,1\}^n$. Then $(\x(C),\y,\nu)$ satisfies the first-order necessary conditions for local
optimality to problem~\eqref{kkt-formulation}, namely~\eqref{opt}, \eqref{feas}, and~\eqref{compl}, if $\beta \ge \frac{|\UU|}2-1$.
\end{proposition}
\begin{proof}
The identities in Lemma~\ref{kktlem} suggest already the choice of dual variables $(\bm{\lambda}, \bm{\mu}, \bm{\psi}, \bm{\theta}, \bm{\phi}, \bm{\delta}, \sigma)\in \Delta^\UU\times\R^n_+\times\R^n_+\times\R^n_+\times\R^n_+\times\R^n_+\times \R$; but first observe that for all
$(i,\Ub)\in C\times \UU$ we have
\begin{align*}
[\nabla_{\x} g_\Ub(\x(C), \y)]_i &= [\Hb_\Ub\x(C)]_i=
2\e_i\T[\beta(\overline \Eb-\Ub)- \Ub]\x(C) \\
& = -2[\Ub\x(C)]_i=x_i(C)-2 =\frac 1{|C|}-2
\end{align*} 
and $$[\nabla_{\y} g_\Ub(\x(C), \y)]_i= - \biggl [\frac{1}{\gamma} \y \biggl ]_i=-
\frac{1}{\gamma} \e_i\T \y = -\frac{1}{\gamma} y_i\,.$$ Since $(\x(C), \y) \in \Delta^V_{\varepsilon}$, then $\supp(\x(C)) = \supp(\y)$ and since $\y \in \{0,1\}^n$, $y_i = 1$ $\forall i \in C$.
Therefore $h_i^x(\x(C)|\bm{\lambda})=\frac 1{|C|}-2$ and $h_i^y(\y|\bm{\lambda})=-\frac 1{\gamma}$  for any $\bm{\lambda}\in \Delta^\UU$ and for all $i\in C$. Moreover, observe that
\begin{align*}
\nu &= G(\x(C), \y)=g_\Ub(\x(C), \y)= \che \x(C)\T\Hb_\Ub\x(C) - \frac{1}{2 \gamma } \norm{\y}^2 \\
&= \frac1{2|C|}-1 - \frac{1}{2 \gamma } |C| \quad\mbox{ for all }\,\Ub\in \UU\, ,
\end{align*}
so that defining $\sigma := \frac1{|C|}-2\in \R$ and $\bm{\theta} := \frac{1}{\gamma} \y \in \R^n_+$, it results $\nu = \che (\sigma - \e\T \bm{\theta})$. We also define $\bm{\phi} = \bm{\delta} = \bm{\psi}  := \oo  \in \R^n$, with $\oo$ the all-zeros vector, and $\bm{\mu}:=\h(\x(C)) - \sigma\e  \in \R^n$,  
so that $\supp(\bm{\mu})\cap C=\emptyset$, where $\h(\x(C)) := \nabla_{\x} g_\Ub(\x(C), \y)$. Indeed, for any $i \in C$, by above reasoning, $h_i(\x(C)) = \sigma$. In this way, it remains to show that $\mu_i\ge 0$  for all $i\in V\setminus C$.
Define
$$ {\rm deg}_\Ub(i|C):=|C|[\Ub\x(C)]_i\le |C|\, ,$$
which counts the edges in $\G(\Ub)$ linking $C$ to an outside vertex $i\in V\setminus C$. Then for all $\Ub\in \UU$ we have
$$[\nabla_{\x} g_\Ub(\x(C), \y)]_i = 2\left[\beta (1-\frac {{\rm deg}_\Ub(i|C)}{|C|})-\frac {{\rm deg}_\Ub(i|C)}{|C|}\right] \, ,$$
which for large enough $\beta> 0$ can be negative only if ${\rm deg}_\Ub(i|C)=|C|$, i.e., if $C\cup\{i\}$ is a clique in $\G(\Ub)$ as well, in which case
$[\nabla_{\x} g_\Ub(\x(C), \y)]_i = -2$ and $[\nabla_{\x} g_\Ub(\x(C), \y)]_i -\sigma = -\frac 1{|C|}$. But since $C$ is a maximal common clique, for all $i\in V\setminus C$ there must be at least one  $\Ub\in \UU$ such that $C\cup\{i\}$ is not a clique in $\G(\Ub)$. In other words, if we decompose
$$\UU = \UU_i^+ \cup \UU_i^-\quad\mbox{ with }\quad \UU_i^+:= \{\Ub\in\UU: {\rm deg}_\Ub(i|C)<|C|\}\, \mbox{ and }\, \UU_i^-:=\UU\setminus \UU_i^+\,,$$
then $\UU_i^+\neq\emptyset$ for all $i\in V\setminus C$, and hence there must be $\bm{\lambda}\in \Delta^\UU$ such that $$\mu_i := \sum_{\Ub\in \UU_i^+} 2 \lambda_\Ub  \left[\beta (1-\frac {{\rm deg}_\Ub(i|C)}{|C|})-\frac {{\rm deg}_\Ub(i|C)}{|C|}-\nu - \frac{1}{2 \gamma} |C| \right]- \frac 1{|C|} \, \sum_{\Ub\in \UU^-_i} \lambda_\Ub \ge 0\, ,$$
if $\beta$ is chosen large enough, which will establish the claim by Lemma~\ref{kktlem}.
Let us now provide
the quantitative estimate for large enough $\beta$, and a suitable choice for $\bm{\lambda}\in \Delta^\UU$. Beforehand note that~\eqref{compl} is satisfied for any $\bm{\lambda}\in \Delta^\UU$ (and above defined $\bm{\mu}$), if $\supp(\bm{\lambda})\subseteq C$.
For any $\Ub\in \UU_i^+$, we have $1-\frac {{\rm deg}_\Ub(i|C)}{|C|} \ge \frac 1{|C|} $ and thus, since $\nu=\frac 1{2|C|}-1- \frac{1}{2 \gamma } |C| $, we have 
\begin{align*}
\beta (1-\frac {{\rm deg}_\Ub(i|C)}{|C|})-\frac {{\rm deg}_\Ub(i|C)}{|C|}- \nu - \frac{1}{2 \gamma } |C| 
&= (1+\beta ) (1-\frac {{\rm deg}_\Ub(i|C)}{|C|}) -\frac 1{2|C|} \\ &\ge \frac 1{|C|}(\che + \beta)
\end{align*}
and therefore
$$\sum_{\Ub\in \UU_i^+} 2 \lambda_\Ub  \left[\beta (1-\frac {{\rm deg}_\Ub(i|C)}{|C|})-\frac {{\rm deg}_\Ub(i|C)}{|C|}-\nu\right]\ge \frac 1{|C|}(1 + 2\beta) \sum_{\Ub\in \UU_i^+}  \lambda_\Ub \, .$$
On the other hand,
$$\frac 1{|C|}\, \sum_{\Ub\in \UU^-_i} \lambda_\Ub = \frac 1{|C|}-  \frac 1{|C|}\sum_{\Ub\in \UU_i^+}  \lambda_\Ub   \, ,$$
so that
$$\frac 1{|C|}(1 + 2\beta)\sum_{\Ub\in \UU_i^+}  \lambda_\Ub -\frac 1{|C|}\, \sum_{\Ub\in \UU^-_i} \lambda_\Ub = \frac 2{|C|}(1 + \beta)\sum_{\Ub\in \UU_i^+}  \lambda_\Ub  -\frac 1{|C|}\, .$$
Multiplying by $|C|$ we arrive at the condition
$$2(1 + \beta)\sum_{\Ub\in \UU_i^+}  \lambda_\Ub \ge 1 \, .$$
As we know $\UU_i^+ \neq\emptyset$ for all $i\in V\setminus C$, we may try with all $\lambda_\Ub=\frac 1{|\UU|}>0$, so that $\sum_{\Ub\in \UU_i^+}  \lambda_\Ub\ge \frac 1{|\UU|}$.
Therefore this condition can be satisfied if $2(1 + \beta)\ge |\UU| $ or $\beta \ge \frac{|\UU|}2-1$. Other choices of $\bm{\lambda}\in \Delta^\UU$ with smaller $\rho=\min\limits_i\min\limits_{\Ub\in \UU^+_i} \lambda_\Ub>0$ may refine this estimate, replacing $|\UU|$ with $1/\rho$. \end{proof}

As is well known, the KKT conditions do, in general, not imply local optimality in nonconvex problems like~\eqref{kkt-formulation}. In the sequel we will show that in contrast to this general situation, here they are indeed sufficient even for local optimality 
in the restricted problem $\min_{(\x, \y)\in \Delta^0_{\varepsilon}} G(\x, \y)$.

Next, given any  common clique $C \in \CC_\UU$, we first consider the restricted problem
\begin{equation}\label{pb_face}
    \begin{split}
        \min & \quad G(\x, \y) \\
        & \text{ s.t. } (\x, \y) \in \Delta^V_{\varepsilon}(C)\,.
    \end{split}
\end{equation}
and will show that the unique optimal solution to problem~\eqref{pb_face} is the characteristic vector of $C$, and this is the only local minimizer over $\Delta^V_{\varepsilon}(C)$. This follows from the following proposition which is an extension of \cite[Proposition 1]{hungerford2019general} to our problem. We also address the KKT condition without invoking constraint qualifications for Equation~\eqref{kkt-formulation}. But before we introduce the (nonempty and finite) set of matrices active at a point $(\x,\y)\in\Delta^V_{\varepsilon}$:
\begin{equation}\label{actmat}
\begin{split}
\UU(\x,\y) &:=\{ \Ub\in \UU : g_\Ub(\x,\y) = G(\x,\y)\} \\
&= \{ \Ub \in \UU : g_\Ub(\x,\y)-g_{\Ub'}(\x,\y) \ge 0\mbox{ for all }\Ub'\in \UU\}\, ,
\end{split}\end{equation}and let us observe the following useful result:

\begin{lemma}\label{lemma1}
   Let  $(\x, \y) \in \Delta^0_{\varepsilon}$. Then
   for all $\Ub\in \UU$ we have
    \item  
    $$g_\Ub (\x, \y) = -\x\T\Ub \x - \frac{1}{2 \gamma} \norm{\y}^2= \che {\norm{\x}}_2^2-1 - \frac{1}{2 \gamma} \norm{\y}^2= G(\x, \y)\, .$$ 
\end{lemma}
\begin{proof}
Since $S= \supp(\x)= \supp(\y)$ is a clique in $\G(\Ub)$, the principal submatrix $\Ub_{S\times S} = \overline\Eb_{S\times S}$. Thus we have $\x\T\Ub\x = (\e\T\x)^2 - \che \x\T\x = 1- \che\x\T\x$. Moreover, \eqref{Hux} implies
\begin{align*}
    g_\Ub (\x) & = \che \x\T \Hb_\Ub\x - \frac{1}{2 \gamma} \norm{\y}^2 \\
    & =  \x\T[\beta (\overline \Eb- \Ub)- \Ub]\x - \frac{1}{2 \gamma} \norm{\y}^2  \\
    & = 0- \x\T\Ub\x - \frac{1}{2 \gamma} \norm{\y}^2 \\
    & = \che\x\T\x - 1 - \frac{1}{2 \gamma} \norm{\y}^2 \, .
\end{align*}
Hence the result.
\end{proof}

\begin{proposition}\label{prop1}
(a) Let $C\in \CC_\UU$ be a common clique. Then there is a unique local minimizer (hence strict and global) of~\eqref{pb_face}, namely $(\x(C), \y)$ with $\y \in \{0,1\}^n$.\\
(b) Moreover, any KKT-point $(\x, \y, \nu)$ of the QCQP reformulation~\eqref{kkt-formulation} with \\*
$\supp(\x) =\supp(\y)=C\in \CC_\UU$ satisfies $\x=\x(C)$, $\y \in \{0,1\}^n$ and \\*$\nu=G(\x(C), \y)= \frac1{2|C|}-1 - \frac{1}{2 \gamma}|C|$, and $C$ is a maximal common clique.
\end{proposition}

\begin{proof}
Since $C$ is a common clique across $\UU$,
we see by Lemma~\ref{lemma1} that $$g_\Ub(\x,\y)= G(\x,\y) = \che \x\T\x - (\e\T\x)^2 - \frac{1}{2 \gamma} \y\T\y = \che \x\T\x -1- \frac{1}{2 \gamma} \y\T\y\,,$$ for all $(\x, \y) \in \Delta^V_{\varepsilon}(C)$ and across all $\Ub\in \UU$. Thus $G(\x, \y)$ is strictly convex with respect to $\x$ and strictly concave with respect to $\y$ and additively separable with respect to $\x$ and $\y$, that is we have $G(\x,\y)=G_{\x}(\x)+ G_{\y}(\y)$, with $G_{\x}(\x)=\che \x\T\x -1$ and $G_{\y}(\y)=- \frac{1}{2 \gamma} \y\T\y\ $. Considering that $\supp(\x)=\supp(\y)\subseteq C$, we have that the vector $\y$ such that $\y_i=1$ for all $i\in C$ and zero otherwise is the unique global minimizer for the function $G_{\y}(\y)$ when considering points $\y \in \Yb_C= \{\y\in \R^n_+:\  \supp(\y)\subseteq C,\ \max_j y_j\leq 1\}$, and $\x(C)$ is the unique global minimizer  for the function $G_{\x}(\x)$ when considering points $\x \in \Xb_C= \{\x\in \Delta^V: \supp(\x)\subseteq C \}$.  If we hence consider the point $(\x(C),\y)$ it is easy to see that this is a strict global minimizer for $G(\x,\y)$, when considering points $(\x,\y) \in \Xb_C\times \Yb_C$. Since  $(\x(C),\y)\in \Delta^V_{\varepsilon}(C)$ and $\Delta^V_{\varepsilon}(C)\subset \Xb_C\times \Yb_C$, assertion~(a) is proved.
Next we establish claim~(b); note that for all $(\x,\y,\d_x, \d_y)\in \Delta^V_{\varepsilon}(C)\times\R^n\times\R^n$ with $\supp(\d_x)\subseteq C$ and $\supp(\d_y)\subseteq C$, also the first-order expression $$(\d_x, \d_y)\T\nabla_{\x, \y} g_\Ub(\x, \y) = \lim_{t\searrow 0} \frac 1t [ g_\Ub(\x+t\d_x, \y + t \d_y)- g_\Ub(\x, \y) ] = \d_x\T\x- \frac{1}{\gamma} \d_y\T \y$$ is the same across all $\Ub$.
We now will directly show that for any $(\x, \y)\in \Delta^V_{\varepsilon}(C)\setminus\lk(\x(C), \overline{\y})\rk$, where $\overline{\y} \in \{0,1\}^n$ with $\supp(\overline{\y}) = C$, there is, at the point $(\x,\y,\nu)$ with $\nu= \che \x\T\x -1 -\frac{1}{2\gamma} \y\T\y=G(\x, \y)$, a direction $(\d_x,\d_y,\zeta)\in \R^n\times\R^n\times \R$ with $\supp(\d_x)\subseteq C$ and $\supp(\d_y)\subseteq C$ which is (first-order) strictly improving, i.e., $\zeta < 0$, and first-order strictly feasible, i.e.,
$$(\d_x\T,\d_y\T,\zeta)\T\nabla_{\x,\y,\nu} (g_\Ub (\x, \y)-\nu) = \d_x\T\x-\frac{1}{\gamma} \d_y\T\y - \zeta < 0\quad \mbox{for all }\Ub\in \UU(\x, \y)\, ,$$
as well as $\supp (\x+\d_x)=C$ and $\supp (\y+\d_y)=C$. Indeed, consider $\d_x= \x(C)- \x$  and $\d_y = \overline{\y} - \y$. 
Not both of them can equal the zero vector simultaneously, and 
$$\d_x\T\x= \x(C)\T\x-\x\T\x = \frac 1{|C|}-\x\T\x =\x(C)\T\x(C) -\x\T\x\le0$$
as well as $$\d_y\T\y= \overline{\y}\T\y-\y\T\y \ge0\, ,$$
so that 
$ \zeta:= \frac 12[\d_x\T\x-\frac{1}{\gamma} \d_y\T\y] <0$ as $\x\neq\x(C)$ or $\y\neq\bar\y$ (or both).
Hence
$\d_x\T\x-\frac{1}{\gamma} \d_y\T\y - \zeta < 0$.
Therefore $\x\neq\x(C)$ or $\y\neq\overline{\y}$ implies existence of a solution $(\d_x,\d_y,\zeta)\in\R^n\times\R^n\times \R$ to the system of strict linear inequalities and equations
\begin{equation}\label{strine}
\begin{array}{rcl}
\zeta &< & 0\\
(\d_x\T,\d_y\T, \zeta)\T\nabla_{\x,\y,\nu} (g_\Ub (\x, \y)-\nu) &< &0 \quad\mbox{ for all }\Ub\in \UU(\x,\y)\\
\e_i\T{\d_x}  &> &0 \quad\mbox{ for all } i\in V\setminus\supp(\x)\\
\e\T\d_x &= &0\\
\e_i\T{\d_y}  &> &0 \quad\mbox{ for all } i\in V\setminus\supp(\y)\, .
\end{array}
\end{equation}
The constraints ${d_x}_i>0$ and ${d_y}_i >0$ are obvious for $i\in C\setminus\supp(\x)=C\setminus\supp(\y)$; a continuity argument also validates it for $i\in V\setminus C$.
Anyhow, solvability of~\eqref{strine} implies by Gordan's theorem of the alternative that $(\x,\y,\nu)$ cannot satisfy the Fritz-John conditions and therefore cannot be a KKT point. Note that KKT conditions imply the Fritz-John conditions without any constraint qualifications, they are needed only for the reverse implication. Therefore $(\x(C), \overline{\y})$ with $C\in \CC_\UU$ is the only possible KKT point of~\eqref{kkt-formulation}. 
If $C\in \CC_\UU$ is a common clique but not maximal, then there is an $i\in V\setminus C$ such that $C'=C\cup \{i\}$ is also a common clique. While $(\x(C), \overline{\y})$ is a KKT point for~\eqref{kkt-formulation}, it cannot be one for the QCQP equivalent to $\min \lk G(\x,\y) : (\x, \y)\in\Delta^V_{\varepsilon}(C')\rk$, as $(\x(C'), \overline{\y'})$ is the only one KKT point for the latter problem (with $\overline{\y'} \in \{0,1\}^n$ chosen such that $\supp(\overline{\y}) = C'$). We calculate the first-order improvement
for the  feasible directions
\begin{align*}
\d_x&=\x(C')-\x(C) = {\textstyle\frac1{|C|+1}}\, \e_i - \sum_{j\in C} \left ({\textstyle \frac1{|C|}-\frac 1{|C|+1}} \right )\e_j \\
&=  {\textstyle\frac 1{|C|+1}}\e_i - \sum_{j\in C} \left ({\textstyle \frac1{|C|(|C|+1)}} \right )\e_j\, ,
\end{align*}
$$\d_y = \overline{\y'}  - \overline{\y} = \e_i \,.$$
The  first-order improvement $\d_x\T\x(C)$ along this $\d_x$ at $\x(C)$ equals, as above, 
$$\d_x\T\x(C)= 0 - \sum_{j\in C} \left ({\textstyle \frac 1{|C|^2(|C|+1)}} \right ) <0\, ,$$ and the first-order improvement $\d_y\T\overline{\y}$ along this $\d_y$ at $\overline{\y}$ equals
$$\d_y\T \overline{\y} = 0\,.$$
Hence $(\x(C),\overline{\y})$ cannot be a local solution to $\min \lk G(\x, \y) : (\x,\y)\in\Delta^V_{\varepsilon} (C')\rk$. Neither can it satisfy the KKT conditions for the larger problem as well.
Note that again, we did not need constraint qualifications as we used a direct first-order descent argument. \end{proof}

On the other hand, any characteristic vector $\x(C)$ and any vector $\overline{\y} \in \{0,1\}^n$ such that $\supp(\overline{\y}) = C$ based upon a {\bf maximal} common clique for $\UU$, is indeed satisfying the KKT conditions, as shown in Proposition~\ref{foco}.
Extending the result~\cite[Proposition 2]{hungerford2019general}, we now will show that it even is a local minimizer of $G(\x,\y)$ over $\Delta^0_{\varepsilon}$
in the smooth epigraphic formulation, i.e., problem~\eqref{kkt-formulation} restricted to $\Delta^0_{\varepsilon}$:
\begin{equation}\label{pb_delta0}
\min_{(\x,\y,\nu)\in \Delta^0_{\varepsilon}\times \R} \{ \nu : g_\Ub(\x,\y)\le \nu\mbox{ for all }\Ub\in \UU\}\, .\end{equation} The proof of Proposition~\eqref{prop2} is a verbatim repetition of the one given in the paper \cite{hungerford2019general}, we reported it here for completeness.

\begin{proposition}\label{prop2}
A point $(\x,\y) \in \Delta^0_{\varepsilon}$ is a local minimizer of \eqref{pb_delta0} if and only if $\x = \x(C)$ and $\y = \overline{\y} \in \{0,1\}^n$  for some maximal common clique $C\in \CC_\UU$. Moreover, every local minimizer is strict.
\end{proposition}

\begin{proof} 
Let $(\x, \y)$ be a local minimizer of \eqref{pb_delta0}, then, since $(\x,\y) \in \Delta^0_{\varepsilon}$, there exists some maximal common clique $C\in \CC_\UU$ such that $(\x,\y) \in \Delta^V_{\varepsilon}(C)$.  $(\x,\y)$ is a local minimizer of $\eqref{pb_delta0}$, it is also a local minimizer of \eqref{pb_face}, which implies $\x = \x(C)$ and  $\y=\overline{\y} \in \{0,1\}^n$ by Proposition~\eqref{prop1}(a).
    
On the other hand, let $C$ be a maximal common clique $C\in \CC_\UU$  and suppose, by way of contradiction, that $(\x(C),\overline{\y})$ is not a local minimizer of \eqref{pb_delta0}. Then, for every $k \in \mathbb{N}_+$, there exists some $(\x^k, \y^k) \in \Delta^0_{\varepsilon}$ with $0 < || (\x^k, \y^k) - (\x(C)-\overline{\y})||_2 < 1/k$ such that $G(\x^k, \y^k) \leq G(\x(C),\overline{\y})$. Because there are only finitely many sets in the unions in \eqref{delta0}, there must exist some common clique $C'$ and some subsequence $(\x^{k_l}, \y^{k_l})_{l=1}^{\infty} \subseteq (\x^{k}, \y^k)_{k=1}^{\infty}$ such that $(\x^{k_l}, \y^{k_l}) \in \Delta^V_{\varepsilon}(C')$ for $l \geq 1$, with $(\x^{k_l}, \y^{k_l}) \to (\x(C),\overline{\y})$. Hence, $(\x(C),\overline{\y}) \in \overline{\Delta^V_{\varepsilon}(C')} = \Delta^V_{\varepsilon}(C')$, which implies $C = \supp(\x(C)) \subseteq C'$. Because $C$ is maximal, we must have $C = C'$, and thus $(\x^{k_l}, \y^{k_l}) \in \Delta^V_{\varepsilon}(C')=\Delta^V_{\varepsilon}(C) $ for each $l \geq 1$. Thus, $(\x(C),\overline{\y})$ is not a strict local minimizer of \eqref{pb_face}, contradicting Proposition~\eqref{prop1}(a).
\end{proof}

\subsection{Going global: exact penalty parameter values}
\label{subsec:global}

The following results are, to some extent, related to~\cite[Proposition 3, Corollary 1]{hungerford2019general}.  
\begin{proposition}\label{maximum}
    If $\{ C_1,C_2\}\in \CC_\UU$ are common cliques, then $$|C_1| \leq |C_2| \iff G(\x(C_1),\y_1) \geq G(\x(C_2),\y_2)\,,$$ where $\y_1 \in \{0,1\}^n$ with $\supp(\y_1) = C_1$ and $\y_2 \in \{0,1\}^n$ with $\supp(\y_2) = C_2$.
Furthermore, a point $(\x, \y) \in \Delta^0_{\varepsilon}$ is a global minimizer of \eqref{pb_delta0} if and only if $\x = \x(C)$ and $\y \in \{0,1\}^n$  for some maximum common clique $C\in \CC_\UU$.
\end{proposition}
\begin{proof}
    First of all we observe that for all $C\in \CC_\UU$,  the characteristic vector  $\x(C)$ and the vector $\y \in \{0,1\}^n$ with $\supp(\y) = C$ such that $(\x(C), \y) \in \Delta^0_{\varepsilon}$ satisfy by Lemma~\eqref{lemma1}
   $$G(\x(C), \y) = \che \x(C)\T\x(C) - 1 -\frac{1}{2 \gamma} \y\T \y =   \frac{1}{2 |C|} -1- \frac{1}{2 \gamma} |C|\, .$$
   The first claim follows. To show the second, let $(\x, \y) \in \Delta^0_{\varepsilon}$ be a global minimizer of \eqref{pb_delta0}. Then  $(\x, \y)$ is also a local minimizer of \eqref{pb_delta0}, hence $\x = \x(C)$ and $\y \in \{0,1\}^n$ with $\supp(\y) = C$ for some maximal common clique $C\in \CC_\UU$ by Proposition~\eqref{prop2}. Further, by global optimality of $(\x, \y)$, we have  $G(\x, \y) \leq G(\mathbf{\tilde{x}}, \mathbf{\tilde{y}})$ for every local minimizer $(\mathbf{\tilde{x}}, \mathbf{\tilde{y}}) \ne (\x, \y)$, hence, again by Proposition~\eqref{prop2},  $G(\x(C), \y) \leq G(\x(\overline{C}), \mathbf{\bar{y}})$ with $\mathbf{\bar{y}} \in \{0,1\}^n $ with $\supp(\mathbf{\bar{y}}) = \overline{C}$ for every maximal common clique $\overline{C} \ne C$. The first claim establishes the result.
\end{proof}

Before we dive into the analysis of the problem \eqref{original-v2}, we consider a restricted version of it with $\y\in \{0,1\}^n$, that is

\begin{equation}\label{original-v2-res}
    \begin{split}
        -\min & \quad G(\x, \y) \\
         \text{ s.t. } &\quad (\x,\y)\in \Delta^V_{\varepsilon}\\&\quad \y \in \{0,1\}^n\,,
    \end{split}
\end{equation}
and give a first equivalence result that connects global minima of the restricted problem and maximum common cliques over our family of graphs. So we have 
\begin{theorem}\label{Th:intequiv}
   For any given  $\varepsilon \in \left(0, \frac{1}{n}\right]$ and  $\gamma > 0$, 
   a point $(\x, \y)$  is a global minimizer of \eqref{original-v2-res} with
    $$\beta > \biggl ( 1 - \frac{1}{2 (n-1)} + \frac{1}{2 \gamma} n \biggl) \frac{1}{2\varepsilon^2} ,\hspace{0.2cm}  $$ 
and $n$ the order of the graphs $\G(\Ub)$,
if and only if $\x = \x({C})$ for some maximum common clique $C \in \CC_{\UU}$. 
\end{theorem}
\begin{proof}
 We will show that if, for any given  $\varepsilon \in \left(0, \frac{1}{n}\right]$,
    $(\x, \y)$ is a global minimizer of \eqref{original-v2-res}  then $(\x, \y)$ is in $\Delta^0_{\varepsilon}$. We hence assume, by contradiction, that $(\x, \y) \in \Delta^V_{\varepsilon}$ is a global minimizer for \eqref{original-v2-res}, but $(\x, \y) \notin \Delta^0_{  \varepsilon}$. In this case, $C := \supp(\x) = \supp(\y)$ is not a common clique. In particular, there exists $\tilde{\Ub} \in \UU$ such that $C$ is not a clique in $\G(\tilde{\Ub})$, hence $[\tilde \Ub]_{hk}=0$ for some $\lk h,k\rk \in \binom{C}{2}$. We have
    \begin{align*}
    G(\x, \y) & \ge g_{\tilde{\Ub}}(\x,\y) =
          \beta \x\T  (\overline\Eb -\tilde{\Ub} ) \x - \x\T \tilde{\Ub} \x - \frac{1}{2 \gamma} \y\T \y\\
        & = - \x\T  \tilde{\Ub} \x+ \beta \sum_{\{i, j \}\in \binom {C}{2}} (1 - [\tilde{\Ub}]_{i j})x_i x_j -  \frac{1}{2 \gamma} \y\T \y \\
        & \geq   - \x\T  \tilde{\Ub} \x+ 2  \beta (1 - [\tilde{\Ub}]_{h k})x_h x_k -  \frac{1}{2 \gamma} \y\T \y  = - \x\T  \tilde{\Ub} \x+ 2 \beta x_h x_k  -  \frac{1}{2 \gamma} \y\T \y \\
        & \geq - \max_{\x \in \Delta} \x\T \tilde{\Ub}\x + 2 \beta x_h x_k -  \frac{1}{2 \gamma} \y\T \y \\
        & =  \frac{1}{2 \omega(\tilde{\G})} -1 + 2 \beta x_h x_k -  \frac{1}{2 \gamma} \y\T \y \\
        & \geq \frac{1}{ 2 (n-1)} - 1 + 2 \beta x_h x_k  -  \frac{1}{2 \gamma} \y\T \y \\
        & \geq \frac{1}{ 2 (n-1)} - 1 +  2 \beta \varepsilon^2 y_h y_k -  \frac{1}{2 \gamma} |C| \,,
\end{align*}
where $\omega(\tilde{\G})\le n-1$ is the  clique number of $\G(\tilde{\Ub})$, which cannot be complete as the edge $\lk h,k\rk$ is missing for sure. In the last row, we used the fact that $(\x, \y) \in \Delta^V_{ \varepsilon}$, then $x_i \geq  \varepsilon y_i$ for all $i\in V$ and $\y\T \y = |C|$. In particular,  $y_h = y_k = 1$, hence $$G(\x, \y) > \frac{1}{ 2 (n-1)} - 1 +  2 \beta  \varepsilon^2 -  \frac{1}{2 \gamma} |C| \,,$$ and since $$\beta > \biggl(1 - \frac{1}{2(n-1)} + \frac{1}{2\gamma}n \biggl) \frac{1}{2  \varepsilon^2} \geq \biggl(1 - \frac{1}{2(n-1)} + \frac{1}{2\gamma}|C| \biggl) \frac{1}{2  \varepsilon^2} \,,$$  
we obtain $G(\x, \y) > 0$.
Let $(\x_0, \y_0) \in \Delta^0_{ \varepsilon}$; such a point always exists (indeed in the worst scenario $\supp(\x_0) = \supp(\y_0)$ is a  singleton, which is always a common clique). Then for all $\Ub \in \UU$  we have from Lemma~\eqref{lemma1} and the fact that $(\x_0, \y_0)\in \Delta^V_{  \varepsilon}$ implies $\x_0\T\x_0 \le 1$,
\begin{align*}
G(\x_0, \y_0) &= g_\Ub(\x_0, \y_0) =  - \x_0\T \Ub\x_0 - \frac{1}{2 \gamma} \y_0\T \y_0 \\
&= \frac{1}{2} \x_0\T \x_0 - 1 - \frac{1}{2 \gamma} \norm{\y_0}^2< 0 \text{ for all } \Ub \in \UU\,.
\end{align*}
We thus arrive at the  absurd relation $G(\x_0, \y_0) < 0 < G(\x, \y)$, contradicting global optimality  of $(\x, \y)$ for the problem \eqref{original-v2-res}. Hence the result.
\end{proof}

We now define the concept of  stable global minimum, which will be useful when analyzing the theoretical properties of  our problem \eqref{original-v2}. Note that stability in this sense is not related to the notion of stable sets in graph theory, but rather has connotation with parametric (continuous) optimization.

\begin{definition}\label{def:stablegm}
A point $(\x,\y)$ is a \emph{stable global minimizer} of problem \eqref{original-v2} if for any fixed $\varepsilon \in \left(0, \frac{1}{n}\right]$ there exists a $\bar \delta>0$ such that $\varepsilon_{\bar \delta}=\varepsilon-\bar \delta>0$
and $(\x,\y)$ is a global minimum of problem \eqref{original-v2} for all pairs $(\varepsilon_{\delta}, \beta_{\varepsilon_{\delta}})$, with   $\beta_{\varepsilon_{\delta}}$ suitably chosen penalty parameter, $\gamma>0$ and $\delta\in[0, \bar\delta]$.
  
\end{definition}
We thus can finally state an equivalence result between the stable global minima of the continuous problem \eqref{original-v2} and the maximum common cliques.

\begin{theorem}\label{equivstable}
 For any given  $\varepsilon \in \left(0, \frac{1}{n}\right]$ and $\gamma > 0$, 
   a point $(\x, \y)$  is a stable global minimizer of \eqref{original-v2} with
    $$\beta_{\varepsilon_{\delta}} > \biggl ( 1 - \frac{1}{2 (n-1)} + \frac{1}{2 \gamma} n \biggl) \frac{1}{2\varepsilon_{\delta}^2} ,\hspace{0.2cm}  $$ 
     and $n$ the order of the graphs $\G(\Ub)$,
if and only if $\x = \x({C})$ for some maximum common clique $C \in \CC_{\UU}$ and $\y \in\{0,1\}^n$. 
 
\end{theorem}
\begin{proof}
 For any given  $\varepsilon \in \left(0, \frac{1}{n}\right]$, we need to show that there exists a $\bar \delta>0$ such that $\epsilon_{\bar \delta}>0$ and $(\x,\y) \in \Delta^0_{    \varepsilon_{\delta}}$ for all choices of $\delta \in [0,\bar \delta].$ The proof will then follow from Proposition~\eqref{maximum}. Suppose by way of contradiction that $(\x, \y) \notin  \Delta^0_{\varepsilon_{\delta}} $ for all  $\delta<\bar\delta$ sufficiently small. So that $C:=\supp(\x)=\supp(\y)$ is not a common clique. Before we establish the  claim on $C$, we show integrality of $\y$, again by contradiction, so suppose some $y_i$ are fractional (then $i\in C$). For this $\y$ we consider the following scenario:
  if for some  sufficiently small $\delta$, we have $x_i> \eps_{\delta} y_i$  with fractional $y_i$, then $\d=(\oo,\e_i)$ is a feasible direction at $(\x,\y)$ which is strictly improving. Indeed,
 let $\Ub_t  \in \UU(\x, \y + t\d)$, i.e.,  $g_{\Ub_t}(\x, \y + t\d) = G(\x, \y + t\d)$.
We have 
\begin{align*}
    G(\x, \y) & \geq  g_{\Ub_t} (\x,\y)  = -\x\T \Ub_t\x + \beta \x\T   (\overline\Eb -\Ub_t) \x - \frac{1}{2 \gamma} \y\T \y \\
    & > -\x\T \Ub_t\x + \beta \x\T   (\overline\Eb -\Ub_t) \x - \frac{1}{2 \gamma} \y\T \y -  \frac{1}{2 \gamma} ( t^2  + 2 t y_i)  \\
    & = -\x\T \Ub_t\x + \beta \x\T   (\overline\Eb -\Ub_t) \x - \frac{1}{2 \gamma} (\y\T \y + t^2  + 2 t y_i) \\
    & = -\x\T \Ub_t\x + \beta \x\T   (\overline\Eb -\Ub_t) \x - \frac{1}{2 \gamma} (\y + t \d)\T(\y + t \d) \\
  &  = g_{\Ub_t}(\x, \y + t\d) = G(\x, \y + t\d)\,.
\end{align*}
Since $(\x, \y)$ is a stable global  minimizer of~\eqref{original-v2}, we obtain a contradiction. 

Now that we have proved $\y \in\{0,1\}^n$, we turn to the claim that $C$ is a common clique, and again argue by contradiction: if this were not true, then
 there exists $\tilde{\Ub} \in \UU$ such that $C := \supp(\x)$ is not a clique in $\G(\tilde{\Ub})$, hence $[\tilde \Ub]_{hk}=0$ for some $\lk h,k\rk \in \binom{C}{2}$. The proof is a verbatim repetition of the proof given in Theorem \ref{Th:intequiv}, where we consider again a $\delta<\bar\delta$ sufficiently small. Hence the result.
\end{proof}

\section{A projection-free algorithm for the continuous min-max reformulation}
\label{section:alg}
In this section, we present a first-order, projection-free algorithm specifically designed for our problem. In particular, we develop a new subdifferential notion based on the classical Clarke and Goldstein subdifferentials discussed in \cite{bagirov2014introduction} and study its theoretical properties. 
Afterwards, we implement a Frank-Wolfe algorithm that takes advantage of this new subdifferential, proving sublinear convergence rates for it. To the best of our knowledge, this is the first attempt in the literature to prove convergence rates for the non-convex case. In the convex (or concave max-min) case, a Frank-Wolfe method was introduced and analyzed in~\cite{White}, but without giving any convergence rate.

\subsection{Preliminaries}
\label{subsec:preliminaries}
Before introducing the algorithm, we clarify how first-order information of the nonsmooth function 
$G$ is exploited. Since $ G $ is the pointwise maximum of a finite set of locally Lipschitz continuous functions $ g_\Ub $, it is itself Lipschitz continuous in the domain $\Delta^V_{\varepsilon}$, with Lipschitz constant $L$ equal to:
\begin{equation}
\label{eq:L}
L := \max_{\substack{\Ub \in \UU, \\ (\x,\y) \in \Delta^V_{\varepsilon}}} \left\| \left( 2 \left( - (1+\beta) \Ub + \beta \overline \Eb \right) \x,\; -\frac{1}{\gamma} \y \right) \right\| > 0 \,.    
\end{equation}
This property allows us to apply the generalized subdifferential framework for Lipschitz functions introduced in \cite{bagirov2014introduction}. In the present context, the constant $L$ dominates the Lipschitz constants of the gradients related to the quadratic functions $g_\Ub$ defining $G$. As a result, even though $L$ is defined as a function Lipschitz constant, it can be safely substituted for the gradient Lipschitz constant of the functions $g_\Ub$ throughout the analysis.

To understand the subdifferential notion adopted in the algorithm, it is necessary to clarify which set of matrices is supposed to be available at each iteration. 
In particular, knowing the set of active matrices would allow us to calculate the Clarke subdifferential, a standard tool in convex analysis. 

\begin{definition}[Clarke subdifferential]
The \emph{Clarke subdifferential} of $ G $ at a point $ (\x,\y) $, denoted by $ \partial G(\x,\y) $, can be expressed as:
$$
\partial G(\x,\y) := \conv \lk \nabla g_\Ub(\x,\y) : \Ub\in \UU(\x,\y) \rk.
$$
\end{definition}

This expression follows directly from the standard characterization of the Clarke subdifferential of a pointwise maximum of finitely many continuously differentiable functions (see, e.g.,~\cite[Theorem 3.23 and Corollary 3.5]{bagirov2014introduction}). In fact, $
\partial G(\x,\y)$ is a (nonempty) polytope, because $\UU(\x,\y)$ is finite.

However, an algorithm that relies only on the Clarke subdifferential to solve problem ~\eqref{original-v2} cannot provide meaningful theoretical convergence results due to the non-differentiability of the function 
$G$. Therefore, we assume that for any point $ (\x,\y) $, the algorithm has access to the set $ \UU^\delta(\x,\y) $ approximating the active set with a tolerance $ \delta $, in a sense to be defined below, rather than the exact active set $ \UU(\x,\y) $. This assumption reflects realistic scenarios where exact function evaluations or full knowledge of the underlying structure are unavailable. From a computational perspective, this relaxation is also justified because the exact maximizers of $ G(\x,\y) $ may not be numerically distinguishable when multiple $g_\Ub$ yield similar values. 

\begin{definition} [Approximately active matrices]
Let $\delta > 0$ be a small parameter and $(\x,\y)\in \Delta^V_\epsilon$.\\
Recall the (nonempty and finite) set of active matrices at a point $\UU(\x,\y)$ defined in~\eqref{actmat}.
Now define
the set of approximately active matrices (again nonempty and finite) at the same point as
$$
\UU^\delta(\x,\y) := \lk \Ub \in \UU : g_\Ub(\x,\y) \geq G(\x,\y) - \delta \rk\, .
$$
\end{definition}

Approximating the active set with a tolerance $ \delta $ therefore provides a stable and implementable mechanism; to this end, we will now construct a new approximate  set of active matrices $\UU_\eta(\x,\y) \subseteq \UU^\delta(\x,\y)$ which will contain only  matrices in $\UU^\delta(\x,\y)$ that become active within a certain neighborhood of $(\x,\y)$. 

To be precise, the process to obtain $\UU_\eta(\x,\y)$ is the following. Given a point $ (\x,\y) $, use as neighborhood the Euclidean ball $ B_\eta(\x,\y) $ of radius $\eta \geq 0$ around $(\x,\y)$ and take $ \eta \leq \frac{\delta}{2L} $, then the only functions that may become active in $ B_\eta(\x,\y) $ are those in $ \UU^\delta(\x,\y) $ (the result is straightforward due to the Lipschitz continuity of the functions $g_\Ub$ and $G$). Finding exactly which $ \delta $-active functions become active for at least  one point $ (\ss,\t) $ within $ B_\eta(\x,\y) $ would require solving a non-linear optimization problem, with a high computational cost. For this reason, as proposed in \cite{Goldstein1977}, we decide to approximate each function $ g_\Ub $ using its first-order Taylor expansion around $ (\x,\y) $, which produces a reliable estimate as long as $ \eta $ is sufficiently small. 
To this end, let us consider the following system of linear inequalities in $(\ss,\t)$, focussing on a certain fixed $\bar\Ub\in \UU^\delta(\x,\y)$:

\begin{equation}
\label{system:my_B}
\begin{cases}
\begin{aligned}
\langle \nabla g_{\Ub}(\x,\y) - \nabla g_{\bar\Ub}(\x,\y), (\ss,\t) - (\x,\y) \rangle \leq & g_{\bar\Ub}(\x,\y) - g_{\Ub}(\x,\y) + L\eta^2, \\
& \forall \Ub \in \UU^\delta(\x,\y)\, ,
\end{aligned} \\
\norm{(\ss,\t) - (\x,\y) }_\infty \leq \frac{\eta}{\sqrt{n}} \, , \\
(\ss,\t) \in \Delta_\epsilon^V\, .
\end{cases}
\end{equation}

The first constraint tests whether $ g_{\bar \Ub} $ becomes approximately active under the linearized model, while the remaining constraints ensure that the perturbed point $ (\ss,\t) $ remains within $ B_\eta(\x,\y) \cap \Delta_\epsilon^V $. 

\begin{definition}[approximate local active set]
\label{def:approx_active_set}
Given parameters $\delta > 0$ and $\eta \leq \frac{\delta}{2L}$, and a point $(\x,\y)\in\Delta_\epsilon^V$, we define the \emph{approximate local active set} at $(\x,\y)$ as

$ \UU_\eta(\x,\y) := \left\{\bar\Ub \in \UU^\delta(\x,\y)\ :\ \exists\, (\ss,\t)\in
\Delta_\epsilon^V\ \text{satisfying system~\eqref{system:my_B}} \,\right\}. $
\end{definition}

As in~\cite{Goldstein1977}, we infer that  $\UU(\x,\y)\subseteq \UU_\eta(\x,\y) $ for any $\eta\ge 0$: indeed, taking $ (\ss,\t) = (\x,\y) $ in system~\eqref{system:my_B}, it is immediate to see that all matrices active in $ (\x,\y) $ satisfy the system. In particular we have $\UU_\eta(\x,\y) \neq \emptyset$. 

We can finally define the exact subdifferential used in the algorithm. 

\begin{definition}[$\eta$-subdifferential]
\label{def:eta_subdiff}
$$
\partial_\eta G(\x,\y) = \conv \lk \nabla g_{\bar\Ub}(\x,\y) : \bar\Ub \in \UU_\eta(\x,\y) \rk
$$
\end{definition}

In fact, this set is again a (nonempty) polytope as also $ \UU_\eta(\x,\y)$ is finite.
In Proposition~\ref{prop:Goldstein_subdiff_properties} we summarize some relevant properties regarding the set given in Definition~\ref{def:eta_subdiff}, showing that in fact it has the characteristics of a subdifferential. This result establishes the link between the approximate subdifferential and the Clarke subdifferential, and is key for interpreting both the theoretical analysis given in this section and the numerical results given in the next one. 

\begin{proposition}
\label{prop:Goldstein_subdiff_properties}
We list below a set of useful properties of the $\eta$-subdifferential:
\begin{enumerate}
    \item The mapping $\partial_\eta G$ is upper semicontinuous.
    \item Let $\partial G(\x,\y)$ denote the Clarke subdifferential of $G$. Then:
   $$
       \bigcap_{\eta >0} \partial_\eta G(\x,\y) = \partial_0 G(\x,\y) = \partial G(\x,\y)\, .
    $$
\end{enumerate}
\end{proposition}

\begin{proof}
\begin{enumerate}
\item Let $(\x_k, \y_k) \to (\x, \y)$ and let $\h_k \in \partial_\eta G(\x_k, \y_k)$ be such that $\h_k \to \h$. We must show that $\h \in \partial_\eta G(\x, \y)$.
The relation $\bar\Ub \in \UU_\eta(\x_k,\y_k)$ means that there exists $(\ss_k,\t_k) \in B_\eta(\x_k,\y_k)\cap\Delta_\epsilon^V$ satisfying system~\eqref{system:my_B} at $(\x_k,\y_k)$. Because $\Delta_\epsilon^V$ is compact and $\eta$ is fixed, the sequence $(\ss_k,\t_k)$ admits a convergent subsequence with limit $(\ss,\t)\in B_\eta(\x,\y)\cap\Delta_\epsilon^V$. By continuity of all $g_\Ub$ and their gradients, the inequalities in~\eqref{system:my_B} remain valid in the limit for all $\Ub\in\UU^\delta(\x,\y)$, which implies $\bar\Ub \in \UU_\eta(\x,\y)$. 

Finally, since $\nabla g_{\bar\Ub}(\x_k,\y_k)\to\nabla g_{\bar\Ub}(\x,\y)$ for all $\bar\Ub\in \UU$, and convex combinations are preserved under limits, we obtain $\h \in \conv\{\nabla g_{\bar\Ub}(\x,\y):\bar\Ub\in\UU_\eta(\x,\y)\} = \partial_\eta G(\x,\y)$. Therefore, $\partial_\eta G$ is upper semicontinuous.

\item 

First, it is easy to see that $\partial_0 G(\x,\y) = \partial G(\x,\y)$. 
Indeed, imposing $\eta = 0$ in system~\eqref{system:my_B} forces $(\ss,\t) = (\x,\y)$ by the second inequality. 
Consequently, the first inequality becomes
\[
0 \leq g_{\bar\Ub}(\x,\y) - g_{\Ub}(\x,\y) 
\quad \forall\ \Ub \in \UU^\delta(\x,\y).
\]
The only matrices $\bar\Ub \in \UU$ satisfying this condition are those that are active at $(\x,\y)$, 
which shows that $\UU_0(\x,\y) = \UU(\x,\y)$ and hence $\partial_0 G(\x,\y) = \partial G(\x,\y)$.

We now show that 
\[
\bigcap_{\eta > 0} \partial_\eta G(\x,\y) = \partial_0 G(\x,\y).
\]
It is sufficient to prove that there exists $\bar{\eta} > 0$ such that 
$\partial_\eta G(\x,\y) = \partial_0 G(\x,\y)$ for all $\eta \leq \bar{\eta}$.
Let $\bar{\delta} > 0$ be such that $\UU^{\bar{\delta}}(\x,\y) = \UU(\x,\y)$. 
Such a value exists since each $\Ub$ is either active or strictly inactive, 
in which case $g_\Ub(\x,\y) < G(\x,\y)$, and the difference 
$G(\x,\y) - g_\Ub(\x,\y)$ is strictly positive. 
Setting $\bar{\eta} = \frac{\bar{\delta}}{2L}$, 
the Lipschitz continuity of the functions $g_\Ub$ implies that, 
for all $\eta \leq \bar{\eta}$, the only functions that can become active 
within $B_\eta(\x,\y)$ are those already active at $(\x,\y)$. 
Hence, $\UU_\eta(\x,\y) = \UU_0(\x,\y)$ for all $\eta \leq \bar{\eta}$, 
and therefore $\partial_\eta G(\x,\y) = \partial_0 G(\x,\y)$.
\end{enumerate}
\end{proof}

\subsection{Algorithmic scheme}
\label{subsec:alg_scheme}
Having introduced the theoretical tools required to handle the non-smoothness of $ G $, we now turn to the design of the optimization algorithm. We propose a projection-free first-order method that extends the classical Frank–Wolfe   algorithm to the nonsmooth and adversarial setting by incorporating the sub\-dif\-ferential $ \partial_\eta G$. 
The proposed algorithm is proved to converge to first-order stationary points of $ G $, and in our experiments it empirically converges to maximal common cliques in the underlying backbone networks.

Given that our feasible domain is a polytope, a projection-free optimization strategy, such as the Frank-Wolfe method, is especially suited to this setting \cite{frank1956algorithm}. One of the main advantages of the Frank-Wolfe method is that it avoids costly projections: instead of computing a projection onto the feasible set at each iteration, it indeed solves a linear minimization problem over the same set. This feature makes the method particularly attractive in high-dimensional scenarios. Additionally, the Frank-Wolfe method tends to produce sparse iterates, as each point is a convex combination of a small number of extreme points of the feasible region. The main steps of the algorithm are reported in Algorithm~\ref{Alg. FW}. In our context, we assume that the seeker, at each iteration, can query a \emph{first-order adversary oracle} with a feasible point. The oracle then returns either the matrices or the function gradients that  belong to the approximate local active set at that point, and these are then used to construct the corresponding $\eta$-subdifferential employed by the algorithm.

\begin{algorithm}[ht]
\caption{Frank-Wolfe for the adversarial min-max problem}
\label{Alg. FW}
\begin{algorithmic}[1]
     \State Set parameters $\varepsilon, \, k_{\max}, \, \delta, \, \eta, \, \xi$ and choose an initial point $ (\x_0,\y_0) \in \Delta^V_{\varepsilon} $.
    \For{$k = 0, \dots, k_{\max}$}
        \State Compute $ \partial_\eta G(\x_k,\y_k) $ using the \emph{ first-order adversary  oracle}
        \State Compute $ ((\hat{\x}_k,\hat{\y}_k), \hat{\h}_k) \in \text{LMO}_{\Delta^V_{\varepsilon}}(\partial_\eta G(\x_k,\y_k)) $
        \State Set $ \d_k = (\hat{\x}_k,\hat{\y}_k) - (\x_k,\y_k) $
        \If{$ \hat{\h}_k^\top \d_k \geq -\xi $}  \textbf{STOP}
        \EndIf
        \State Update $ (\x_{k+1},\y_{k+1}) = (\x_k,\y_k) + \alpha_k \d_k $, with $ \alpha_k \in (0,1] $ chosen via an Armijo line search
    \EndFor
\end{algorithmic}
\end{algorithm}

The algorithm starts from a feasible initial point $ (\x_0,\y_0) \in \Delta^V_{\varepsilon} $ and iteratively refines it. At step 3, 
the  first-order adversary oracle is thus queried with the point $ (\x_k,\y_k)$  and computes the set  $\UU_\eta(\x_k,\y_k)$ following the framework presented in the previous subsection, that is by suitably solving system~\eqref{system:my_B} for the $\delta$-active matrices in the point $ (\x_k,\y_k)$. The set $\UU_\eta(\x_k,\y_k)$ is then used to define $ \partial_\eta G(\x_k,\y_k) $. Step 4 invokes a Linear Minimization Oracle (LMO), defined by:

\begin{equation}
\label{eq:LMO}
\min_{(\x,\y) \in \Delta^V_{\varepsilon}} \max_{\h \in \partial_\eta G(\x_k,\y_k)} \h^\top (\x - \x_k, \y - \y_k).
\end{equation}
The LMO is modeled after the classical Frank-Wolfe   algorithm: in the standard setting, the oracle indeed minimizes a linear approximation of the objective function over the original feasible set. Here, we follow the same reasoning, but we replace the gradient with the $\eta$-subdifferential $\partial_\eta G(\x_k, \y_k)$ in the approximation. This generalization  preserves the central intuition of Frank–Wolfe (that is, moving toward a feasible point that minimizes a linear approximation of the objective), so the structure and purpose of the LMO remain fully aligned with the classical step.

Since min-max problems are difficult to solve directly, we reformulate \eqref{eq:LMO} as a minimization problem. In particular, recalling the definition of $\partial_\eta G$ as a convex combination of gradients, it is possible to write $\h$ by baricentric coordinates with respect to the gradients of the active functions in $ \UU_\eta(\x_k,\y_k) $:
$$
\begin{aligned}
\h &= (\h_{\x_k}, \h_{\y_k}) \\
  &= \left( \sum_{\Ub \in \UU_\eta(\x_k, \y_k)} \lambda_\Ub \cdot 2 \left( - (1 + \beta) \Ub + \beta \overline \Eb \right) \x_k, \; -\frac{1}{\gamma} \y_k \right) \, ,
\end{aligned}
$$
where $r = |\UU_\eta (\x_k, \y_k)|$, and $ \bm{\lambda} \in \R^r_+$ satisfies $ \sum_\Ub \lambda_\Ub = 1$.

Therefore, the inner maximization problem reduces to a linear program over the standard simplex, whose optimal value is necessarily attained at one of the simplex’s vertices. The LMO can thus be written as the following LP: 
$$
\begin{aligned}
\min_{\x, \y, \mu} \quad & \mu - \frac{1}{\gamma} \y_k^\top \y + \frac{1}{\gamma} \|\y_k\|^2 \\
\text{s.t.} \quad & (\x,\y) \in \Delta^V_{\varepsilon} \\
& \mu \geq \nabla_\x g_\Ub(\x_k,\y_k)^\top (\x - \x_k) \quad \forall \, \Ub \in \UU_\eta(\x_k, \y_k)
\end{aligned}
$$ 

In order to quantify progress, we define the LMO gap as follows.
\begin{definition}[$\eta$ gap function]
\label{def:eta_gap}
Given a point $(\x,\y)\in \Delta^V_{\varepsilon}$, the $\eta$ gap in $\Delta^V_{\varepsilon}$ is defined as:
$$
c_\eta(\x,\y) := - \min_{(\ss,\t) \in \Delta^V_{\varepsilon}} \max_{\h \in \partial_\eta G(\x,\y)} \h^\top (\ss-\x,\t-\y).
$$
\end{definition}
It holds that $ c_\eta(\x,\y) \geq 0 $ for all $ (\x,\y) \in \Delta^V_{\varepsilon} $.

\noindent
Finally, we compute the step size using an Armijo line search; to this end, given parameters $ \omega \in (0,1) $ and $ \sigma \in (0, \tfrac{1}{2}) $, and starting from a base step size $\omega^0=1$, we test 
the condition
\begin{equation}
\label{eq:Armijo_stepsize}
G((\x_k,\y_k) + \alpha \d_k) \leq G(\x_k,\y_k) - \sigma \alpha c_\eta(\x_k,\y_k)
\end{equation}
for $\alpha = \omega^m$ by increasing the integer $m\ge 0$. We then set $ \alpha_k = \omega^m$ for the smallest $m$  for which~\eqref{eq:Armijo_stepsize} is satisfied. 

\subsection{Theoretical convergence guarantees}
\label{sec:Convergence Results}
\textbf{Note:} In the remainder of this section, the explicit notation $ (\x, \y) $ is no longer necessary for the scope of the proofs. To simplify the presentation, we will instead use a single variable $ \z \in \mathbb{R}^{2n} $ to denote the concatenated coordinates.

The next result provides a convergence rate for Algorithm~\ref{Alg. FW}. In particular, we show that, choosing $\eta\leq\frac{\delta}{2L}$, for any small enough $\delta$, the algorithm will  give an $(\eta,\xi)$-Goldstein stationary point, which we define as follows.
\begin{definition}[\((\eta,\xi)\)-Goldstein stationary point]
A point $\z^\star\in\Delta^V_\varepsilon$ is called an 
\((\eta,\xi)\)-Goldstein stationary point for the minimization of $G$ over 
$\Delta^V_\varepsilon$ if it satisfies the following condition:
\[
    c_\eta(\z^\star) \le \xi.
\]
In other words, $\z^\star$ is \((\eta,\xi)\)-Goldstein stationary whenever the 
LMO certificate $c_\eta(\z^\star)$ does not exceed~$\xi$. 
This notion matches the stopping criterion of Algorithm~\ref{Alg. FW}, which terminates
as soon as $c_\eta(\z_k)\le\xi$.
\end{definition}

More specifically, the algorithm will satisfy the stopping condition and give an $(\eta,\xi)$-Goldstein stationary point after a finite number of iterations, which can be estimated as follows.

\begin{theorem}
\label{thm:alg_conv}
Let $ G(\z^*) = \min\limits_{\z \in \Delta_\epsilon^V} G(\z) $ and let $ \{\z_k\} $ be a sequence generated by Algorithm~\ref{Alg. FW}. Fix $\eta\leq\frac{\delta}{2L}$, where $L$ is the Lipschitz constant of $G$ defined in \eqref{eq:L} and let $ c^*_k = \min\limits_{i\in [0:k]} c_\eta(\z_i) $. Assume that the step size in the algorithm is determined using an Armijo line search with fixed parameter $ \sigma$ as in \eqref{eq:Armijo_stepsize}. Finally, observe that $D:= \text{diam}\left( \Delta^V_\epsilon \right)^2 = n+1>1$, and let $ \rho = \frac{D([G(\z_0) - G(\z^*)]}{\sigma}. $
Then,
$$
c^*_k \leq 
\begin{cases}
\displaystyle \frac{\rho }{(k+1)\eta}. & \text{if } k < \frac{\rho(1-\sigma)}{\eta} - 1\, , \\[10pt]
\displaystyle \sqrt{\frac{\rho}{(k+1)(1-\sigma)\eta}} & \text{otherwise.}
\end{cases}
$$

In particular, 
$$
\lim_{k \to \infty} c^*_k = 0 \quad \text{at a rate of } \mathcal{O}\left(\frac{1}{\sqrt{k\eta}}\right),
$$
that is, Algorithm~\ref{Alg. FW} reaches an \((\eta,\xi)\)-Goldstein stationary point in 
$O(1/(\eta\xi^2))$ iterations.
\end{theorem}

\begin{proof}
First, recall that the quantity $ c_\eta(\z_k) $ corresponds to:
\begin{equation}
\label{eq:eta_gap}
c_\eta(\z_k) = -\nabla g_\Ub(\z_k)^\top \d_k,
\end{equation}
where $ \d_k $ is the chosen descent direction and $ \Ub \in \UU_\eta(\z_k) $ is the approximate active matrix that provides the steepest descent. 

As mentioned before, our LMO selects a direction $ \d_k $ by solving a min-max problem. Since the inner problem maximizes a linear function over the convex hull of gradients $ \nabla g_\Ub(\z_k) $, such a maximum is hence attained at an extreme point of this convex set. Therefore, there must exist a matrix $ \Ub \in \UU_\eta(\z_k) $ such that equality \eqref{eq:eta_gap} holds, by the fundamental theorem of linear programming.

Adopting an Armijo line search, we can write for every $k$:
\begin{equation}   G(\z^*) - G(\z_0) \leq G(\z_{k+1}) - G(\z_0) = \sum_{i=0}^k \left[ G(\z_{i+1}) - G(\z_{i}) \right] \leq - \sigma \sum_{i=0}^k \alpha_i c_\eta(\z_i)\, . \label{eq:armijo_line_search}
\end{equation}
At each iteration $ k $, there are two possibilities:

\textbf{Case 1:} There is a matrix $ \Ub $ common to both $ \UU_\eta(\z_k) $ and $ \UU(\z_{k+1}) $. Then, the following descent condition holds:
\begin{align}
    \label{eq:descent1}
    G(\z_{k+1}) &= g_\Ub(\z_{k+1}) \\
    \label{eq:descent2}
    &\leq g_\Ub(\z_k) + \alpha_k \nabla g_\Ub(\z_k)^\top \d_k + \frac{\alpha_k^2 L}{2} \|\d_k\|^2 \\
    \label{eq:descent3}
    &\leq G(\z_k) - \alpha_k c_\eta(\z_k) + \frac{\alpha_k^2 L}{2} \|\d_k\|^2  \, ,
\end{align}
where \eqref{eq:descent1} follows from the fact that $ \Ub \in \UU(\z_{k+1}) $, while \eqref{eq:descent3} derives from $g_\Ub \le G$.

We can combine the descent condition with the Armijo line search, as done in \cite{Bomz20b}: let $\alpha_k$ be the accepted Armijo step and $\tilde\alpha=\frac{\alpha_k}{\omega}$ the previous (rejected) trial. Since $\tilde\alpha$ was rejected, it must hold:
\[
G(\z_k+\tilde\alpha \d_k)>G(\z_k)-\sigma\tilde\alpha c_\eta(\z_k).
\]
Combined with the descent bound at $\tilde\alpha$, the following relation can be obtained:
$$G(\z_k)-\tilde\alpha c_\eta(\z_k)+\frac{L}{2}\tilde\alpha^2\|\d_k\|^2\,\ge \, G(\z_k+\tilde\alpha \d_k)\, >G(\z_k)-\sigma\tilde\alpha c_\eta(\z_k) \, , $$
which after simple rearrangements yields 
$$ \tilde\alpha>\frac{2(1-\sigma)c_\eta(\z_k)}{L\|\d_k\|^2} \, .$$
Since $\tilde\alpha=\frac{\alpha_k}{\omega}$, accounting for the upper bound at 1, we get: 
\[
\;\alpha_k \ge \min\!\Big\{1,\;\frac{2\omega(1-\sigma)\,c_\eta(\z_k)}{L\|\d_k\|^2}\Big\}.
\]

\textbf{Case 2:} The active set \( \UU(\z_{k+1}) \) does not contain any matrix from \( \UU_\eta(\z_k) \). We now show that in this case $\alpha_k \|\d_k\| = \|\z_{k+1}-\z_k\|>\eta$, therefore obtaining $ \alpha_k> \frac{\eta}{\|\d_k\|} $.

Suppose by contradiction that $ \|\z_{k+1}-\z_k\|\leq \eta $. The hypothesis $\eta \leq \frac{\delta}{2L}$ implies that $ \|\z_{k+1}-\z_k\| \leq \frac{\delta}{2L} $, and byLipschitz continuity of the functions $g_\Ub$, only the functions in $ \UU^\delta(\z_k) $ can be active functions in $\z_{k+1}$, so $\UU(\z_{k+1}) \subseteq \UU^\delta(\z_k) $. Let us choose a matrix $ \bar\Ub \in \UU(\z_{k+1}) $. We will now show that for every $\Ub\in\UU^\delta(\z_k)$ the inequality
\begin{equation}
\label{eq:systemB_1}
\langle \nabla g_{\Ub}(\x,\y) - \nabla g_{\bar\Ub}(\x,\y), (\ss,\t) - (\x,\y) \rangle \leq g_{\bar\Ub}(\x,\y) - g_{\Ub}(\x,\y) + L\eta^2
\end{equation}
holds. This would automatically imply that $\bar\Ub$ satisfies system~\eqref{system:my_B}, so \\* $\bar\Ub \in \UU(\z_{k+1}) \,\cap \, \UU_\eta(\z_k)$, providing a contradiction.

Fix $\bar\Ub\in\UU(\z_{k+1})$ and $\Ub\in\UU^\delta(\z_k)$. By definition of activity at $\z_{k+1}$,
\begin{equation}
\label{eq:activity_Uj}
g_{\bar\Ub}(\z_{k+1}) - g_{\Ub}(\z_{k+1}) \ge 0.
\end{equation}
Apply the first-order Taylor expansion of each $g_\Ub$ around $\z_k$ with the remainder bounded by the Lipschitz continuity of the gradients: for any $\Ub$ there exists a remainder $r_\Ub$ with
$$
g_\Ub(\z_{k+1}) = g_\Ub(\z_k) + \alpha_k\langle \nabla g_\Ub(\z_k),\d_k\rangle + r_\Ub,
\qquad |r_\Ub| \le \tfrac{L}{2}\alpha_k^2\|\d_k\|^2.
$$
Subtract the expansions for $\bar\Ub$ and $\Ub$ and use \eqref{eq:activity_Uj}:
$$
0 \le \big(g_{\bar\Ub}(\z_k)-g_{\Ub}(\z_k)\big)
    + \alpha_k\langle \nabla g_{\bar\Ub}(\z_k)-\nabla g_{\Ub}(\z_k),\d_k\rangle
    + (r_{\bar\Ub}-r_{\Ub}).
$$
Rearranging gives
\begin{equation}
\label{eq:Taylor_with_r}
\alpha_k\langle \nabla g_{\Ub}(\z_k)-\nabla g_{\bar\Ub}(\z_k),\d_k\rangle
\le g_{\bar\Ub}(\z_k)-g_{\Ub}(\z_k) + (r_{\bar\Ub}-r_{\Ub}). 
\end{equation}
Using the remainder bounds $|r_{\bar\Ub}|,|r_{\Ub}|\le \tfrac{L}{2}\alpha_k\|\d_k\|^2$ we get
$$
r_{\bar\Ub}-r_{\Ub} \le |r_{\bar\Ub}| + |r_{\Ub}|
\le L\alpha_k^2\|\d_k\|^2 \le L\eta^2.
$$
Substitute into \eqref{eq:Taylor_with_r} to obtain:
\begin{equation}
\label{eq:Taylor_eta}
\alpha_k\langle \nabla g_{\Ub}(\z_k)-\nabla g_{\bar\Ub}(\z_k),\d_k\rangle
\le g_{\bar\Ub}(\z_k)-g_{\Ub}(\z_k) + L\eta^2.    
\end{equation}
The last inequality coincides with \eqref{eq:systemB_1} for $(\x,\y)=\z_k$ and $(\mathbf{s},\t)=\z_k + \alpha_k\d_k$, providing the required contradiction and therefore the lower bound $ \alpha_k > \frac{\eta}{\| \d_k\|} $.

Combining the two cases and noting that for small enough $\delta$ it holds $\omega > \delta$, we obtain a general lower bound for $ \alpha_k $:
\begin{equation}\label{eq:alpha_k_bound}
\alpha_k \geq \min \lk1, \frac{\eta}{\|\d_k\|}, \frac{\eta(1 - \sigma)c_\eta(\z_k)}{ \|\d_k\|^2} \rk.
\end{equation}
We note that $ \max \{ \|\d_k\|,\|\d_k\|^2\} $ is bounded by the constant $ D>1 $ due to its definition in Theorem~\ref{thm:alg_conv}, then combining this with the previous inequality \eqref{eq:alpha_k_bound}, reversing the inequality \eqref{eq:armijo_line_search} and recalling the definition of $c^*_k$ in the statement of Theorem~\ref{thm:alg_conv}, we obtain:
\begin{equation}
\label{eq:rate_for_c*}
G(\z_0) - G(\z^*) \geq (k+1) \sigma c_k^* \min \lk1, \frac{\eta}{D}, \frac{\eta(1 - \sigma)c_k^*}{D} \rk.
\end{equation}
Now, assuming $ \eta \leq \frac{\delta}{2L} < D $ (reasonable for small $ \delta $), we obtain:
\begin{itemize}
    \item If $ \frac{\eta}{D} <  \frac{\eta(1 - \sigma)c_k^*}{D} $ (and thus $ (1-\sigma) c^*_k > 1 $), then \eqref{eq:rate_for_c*} implies
    $$
    c^*_k \leq \frac{D [G(\z_0) - G(\z^*)]}{(k+1)\sigma\eta}.
    $$
    Observe that setting $ \rho = \frac{D[G(\z_0) - G(\z^*)]}{\sigma} $, we necessarily have $$ 1 < (1-\sigma)\frac{D [G(\z_0) - G(\z^*)]}{(k+1)\sigma\eta} \quad \mbox{and thus}\quad k < \frac{\rho(1-\sigma)}{\eta} - 1\, .$$
    \item Otherwise, $ (1-\sigma) c^*_k \leq 1 $ and~\eqref{eq:rate_for_c*} yields 
    $$
    c^*_k \leq \left(\frac{D [G(\z_0) - G(\z^*)]}{(k+1)\sigma(1-\sigma)\eta}\right)^{1/2} {= \sqrt{\frac\rho{(k+1)(1-\sigma)\eta}}\; .}
    $$
\end{itemize}
This proves the assertion. \end{proof}

This result yields a convergence rate for the stopping criterion of Algorithm~\ref{Alg. FW}. We now show that, if the algorithm is run with decreasing values of $\eta$, there exists at least one sequence converging to a Clarke stationary point.
\begin{definition}[Clarke dual gap]
Given a point $\z\in \Delta^V_\epsilon$, the Clarke dual gap on $\Delta^V_\epsilon$ is defined as:
$$
c(\z) := - \min_{\w \in \Delta^V_\epsilon} \max_{\h \in \partial G(\z)} \h^\top (\w - \z).
$$
\end{definition}

First, it is necessary to recall the definition of \emph{Clarke stationary points} for minimization of Lipschitz continuous functions over a compact set, which is linked to the \emph{Clarke subdifferential}.

\begin{definition}[First-order stationary point]
\label{def:first_order_stat}
A point $ \z^* $ is called a \emph{Clarke stationary point} for a locally Lipschitz function $ G $ to be minimized over the set $\Delta_\epsilon^V$ if

$$
\max_{\h \in \partial G(\z^*)} \h^\top (\w - \z^*)\geq 0 \quad \forall\ \w \in \Delta_\epsilon^V.
$$               
\end{definition}

\begin{theorem}
\label{thm:final_convergence}

Let $ G(\z^*) = \min\limits_{\z \in \Delta_\epsilon^V} G(\z) $ and let $ \{\z_k\} $ be a sequence generated by Algorithm~\ref{Alg. FW}. Assume that the algorithm is run with a sequence $ \eta_k = \frac{\eta_0}{(1+k)^\tau} $, with $ 0 < \eta_0 < \frac{\delta}{2L} $ and $ 0 < \tau < 1$. As in Theorem~\ref{thm:alg_conv}, $L$ is the Lipschitz constant~\eqref{eq:L} of $G$, the step size is determined using an Armijo line search with fixed parameter $ \sigma$ and $D:= \text{diam}\left( \Delta^V_\epsilon \right)^2 = n+1>1$. Fixing $ \rho = \frac{D[G(\z_0) - G(\z^*)]}{\sigma\eta_0} $, let $ c^*_k = \min\limits_{i\in [0:k]} c_{\eta_i}(\z_i) $.
Then,
\begin{equation}
\label{eq:thm4_ck*}
c^*_k \leq 
\begin{cases}
\displaystyle \frac{\rho}{\sum_{i=0}^{k} \frac{1}{(i+1)^\tau}}. & \text{if } k < \sqrt[1-\tau]{\rho(1-\sigma)(1-\tau)+1} - 1\, , \\[10pt]
\displaystyle \sqrt{\frac{\rho}{(1-\sigma)\sum_{i=0}^{k} \frac{1}{(i+1)^\tau}}} & \text{otherwise.}
\end{cases}
\end{equation}

In particular, 
$$
\lim_{k \to \infty} c^*_k = 0 \quad \text{at a rate of } \mathcal{O}\left(k^{\frac{{1-\tau}}{2}}\right).
$$

Moreover, there exists a limit point of the sequence $ \{\z_k\} $ that is a Clarke stationary point of the problem $\min\limits_{\z \in \Delta_\epsilon^V} G(\z)$.
\end{theorem}

\begin{proof}
Following the exact same reasoning as in Theorem~\ref{thm:alg_conv}, we derive a slight modification of~\eqref{eq:rate_for_c*}:
\begin{equation}
\label{eq:rate_for_c*_modified}
G(\z_0) - G(\z^*) \geq \sigma c_k^* \sum_{i=0}^{k}\min\lk1, \frac{\eta_i}{D}, \frac{\eta_i(1 - \sigma)c_k^*}{D} \rk.
\end{equation}
Assuming, as in Theorem~\ref{thm:alg_conv}, that $ \eta_0 \leq \frac{\delta}{2L} < D $, we can simplify the minimum term, obtaining:
$$
G(\z_0) - G(\z^*) \geq \sigma c_k^* \min\lk \frac{\eta_0}{D}, \frac{\eta_0(1 - \sigma)c_k^*}{D}\rk\sum_{i=0}^{k}\frac{1}{(i+1)^\tau}.
$$
We finally obtain the two cases as in Theorem~\ref{thm:alg_conv}:
\begin{itemize}
    \item If $ (1-\sigma) c^*_k > 1 $, then 
    $$
    c^*_k \leq \frac{D [G(\z_0) - G(\z^*)]}{\sigma\eta_0\sum_{i=0}^{k}\frac{1}{(i+1)^\tau}}.
    $$
    \item Otherwise,
    $$
    c^*_k \leq \left(\frac{D [G(\z_0) - G(\z^*)]}{\sigma(1-\sigma)\eta_0\sum_{i=0}^{k}\frac{1}{(i+1)^\tau}}\right)^{1/2}.
    $$
\end{itemize}
Putting $ \rho = \frac{D[G(\z_0) - G(\z^*)]}{\sigma\eta_0} $ gives~\eqref{eq:thm4_ck*}.

Note that, for $\tau=0$, we would retrieve ~\eqref{eq:rate_for_c*}, while for $0<\tau<1$, we have $$\frac{(k+1)^{1-\tau}-1}{1-\tau}\leq \sum_{i=0}^{k}\frac{1}{(i+1)^\tau} \leq \frac{(k+1)^{1-\tau}}{1-\tau}$$ for $k$ sufficiently large, which guarantees the rate of convergence for $c^*_k$. 

To conclude, we prove that there exists a limit point $\z^*$ of the the sequence $\{\z_k\}$ generated by the algorithm such that $c(\z^*) = 0$. By Definition~\ref{def:first_order_stat}, this implies that $\z^*$ is a Clarke stationary point. 3

Since $c^*_k \to 0$, we can extract from the sequence generated by the algorithm a subsequence $\{\z_h\}$ converging to $\z^*$ such that $c_{\eta_h}(\z_h) \to 0$.

Because any matrix $\Ub$ is either active or inactive in $\z^*$, there exists $ \delta^* > 0$ such that $\UU^{\delta^*}(\z^*) = \UU(\z^*)$. 
All functions $g_\Ub$ are Lipschitz continuous with a common constant $L$, hence we may choose a radius $r^* < \frac{\delta^*}{2L}$ such that for every $\z \in B_{r^*}(\z^*)$:
$$
\UU^{\bar{\delta}}(\z) = \UU(\z), \qquad \text{where } \bar{\delta} = \delta^* - 2Lr^* > 0.
$$ 
In particular, no function inactive at $\z^*$ can become active inside $B_{r^*}(\z^*)$: 
$$
\UU(\z) \subset \UU(\z^*) \qquad \forall \, \z \in B_{r^*}(\z^*).
$$ 
We now restrict the subsequence $\{\z_h\}$ to the indices for which $\z_h \in B_{r^*}(\z^*)$ and we prove that for such $\z_h$,
$$
\UU_{\eta}(\z_h) = \UU(\z_h).
$$ 
Fix such an index $h$. Since $\UU^{\bar{\delta}}(\z_h) = \UU(\z_h)$, any function $g_{\bar{\Ub}}$ that is not active at $\z_h$ satisfies
$$
g_{\bar{\Ub}}(\z_h) - g_{\Ub}(\z_h) < -\bar{\delta}, \qquad \forall \, \Ub \in \UU(\z_h)
$$
From system~\eqref{system:my_B}, a matrix $\bar{\Ub}$ belongs to $\UU_{\eta}(\z_h)$ only if
$$
\langle \nabla g_{\Ub}(\z_h) - \nabla g_{\bar\Ub}(\z_h), \z - \z_h \rangle \leq g_{\bar\Ub}(\z_h) - g_{\Ub}(\z_h) + L\eta^2 \leq -\bar{\delta} + L\eta^2 \;\;\; \forall \Ub \in \UU^\delta(\z_h)
$$
Since $\z$ is restricted by $\|\z - \z_h\| \leq \eta$, we can bound the left-hand side as
$$
\left|\langle \nabla g_{\Ub}(\z_h) - \nabla g_{\bar\Ub}(\z_h), \z - \z_h \rangle\right| \leq \|\nabla g_{\Ub}(\z_h) - \nabla g_{\bar\Ub}(\z_h)\|\|\z-\z_h\|\leq 2L\eta \, .
$$
Hence,
$$
\langle \nabla g_{\Ub}(\z_h) - \nabla g_{\bar\Ub}(\z_h), \z - \z_h \rangle \geq -2L\eta \, .
$$
If $\eta$ is chosen such that
$$
2L\eta < \bar{\delta} - L \eta^2 \, ,
$$
no inactive $\Ub$ satisfies the inequality defining $\UU_\eta(\z_h)$. Therefore, for all $\eta \le \eta^* = \sqrt{1 + \frac{\bar{\delta}}{L}}-1$,  we have

$$
\UU_{\eta}(\z_h) = \UU(\z_h) \qquad \text{for all such } \z_h.
$$

We thus obtain the chain of inclusions:
$$
\UU_{\eta}(\z_h) = \UU(\z_h) \subset \UU(\z^*), \quad \forall \, \z_h \in B_{r^*}(\z^*), \, \forall \, \eta \leq \eta^* \, .
$$
Consequently, whenever $\eta_h \le \eta^*$,
$$
c_{\eta_h}(\z_h) = c(\z_h) \geq c(\z^*).
$$
Since the left-hand side converges to $0$, we conclude that $c(\z^*)=0$, as claimed.

This proves that $ \z^* $ satisfies the first-order stationarity condition in the Clarke sense (Definition~\ref{def:first_order_stat}), therefore there exists an accumulation point of the iterates produced by the algorithm that is a Clarke stationary points.
\end{proof}

This result shows that, if Algorithm~\ref{Alg. FW} is executed with progressively smaller values of $\eta$, there exists at least a subsequence converging to a \emph{Clarke stationary point}. It should be noted that this theoretical result does not guarantee convergence to a global or even a local minimum. However, as will be shown in the next chapter, the algorithm empirically converges to local minima in practice. This observation suggests that coupling the method with a multistart or similar global-search strategy may yield solutions related to large common cliques.

\section{Experimental results}
\label{section:experiment}
In this section, we provide an in-depth analysis of the results for the instances derived by modifying a selection of the DIMACS graphs, as summarized in Tables~\ref{table:1}-\ref{table:3}.
Each experiment was conducted as follows. Given a graph, we randomly selected a fixed proportion of edges to form a backbone network. We then generated 50 different graphs by independently adding edges according to a specified probability, forming the set of adjacency matrices $\UU$. The algorithm parameters were fixed, and for each configuration we ran 10 independent experiments, reporting the averaged results.
The ground-truth maximum common clique size was computed using a standard discrete optimization method, which becomes computationally expensive for larger graphs (instances for which the ground-truth computation was computationally infeasible within a reasonable time limit were not included in the experiments).

Specifically, as shown in Tables~\ref{table:1}-\ref{table:3}, for each DIMACS graph we tested nine configurations obtained by varying two parameters: (i) the backbone fraction $ b \in \{0.25, 0.5, 0.75\} $, representing the proportion of edges forming the common submatrix shared across all matrices in $\UU$; and (ii) the edge addition probability $ p \in \{0.25, 0.5, 0.75\} $, defining the probability of adding additional edges outside the backbone to each matrix $\Ub$. For each configuration, we report the parameter values $b$ and $p$, the largest common clique size found (\textbf{max}), the average common clique size over multiple runs (\textbf{mean}), the standard deviation (\textbf{std}), and the ground-truth maximum common clique size (\textbf{real max}). Values in the \textbf{max} column are highlighted in bold when they match the \textbf{real max}. 

\begin{table}
\centering
\caption{Results obtained by Algorithm \ref{Alg. FW} on  instances from DIMACS dataset (part I).}
\label{table:1}
\begin{tabular}{ccccccc}
\hline\noalign{\smallskip}
\textbf{Graph} & $b$ & $p$ & \textbf{max} & \textbf{mean} & \textbf{std} & \textbf{real max}\\
\noalign{\smallskip}\hline\noalign{\smallskip}
\multirow{9}{*}{C125.9}
& 0.25 & 0.25 & \textbf{5} & 4.00 & 0.45 & 5 \\
& 0.25 & 0.5 & 5 & 3.90 & 0.70 & 6 \\
& 0.25 & 0.75 & 5 & 4.00 & 0.63 & 6 \\
& 0.5 & 0.25 & 7 & 6.10 & 0.70 & 9 \\
& 0.5 & 0.5 & 8 & 6.60 & 0.66 & 9 \\
& 0.5 & 0.75 & 7 & 6.20 & 0.75 & 9 \\
& 0.75 & 0.25 & 12 & 10.40 & 1.02 & 14 \\
& 0.75 & 0.5 & 11 & 10.10 & 1.14 & 14 \\
& 0.75 & 0.75 & 12 & 10.50 & 0.81 & 15 \\
\noalign{\smallskip}\hline\noalign{\smallskip}
\multirow{9}{*}{C250.9}
& 0.25 & 0.25 & 5 & 4.10 & 0.30 & 6 \\
& 0.25 & 0.5 & 5 & 4.40 & 0.49 & 6 \\
& 0.25 & 0.75 & 5 & 4.40 & 0.49 & 6 \\
& 0.5 & 0.25 & 7 & 6.50 & 0.50 & 11 \\
& 0.5 & 0.5 & 8 & 7.00 & 0.63 & 10 \\
& 0.5 & 0.75 & 8 & 6.40 & 0.80 & 11 \\
& 0.75 & 0.25 & 15 & 12.60 & 1.11 & 18 \\
& 0.75 & 0.5 & 16 & 12.20 & 1.47 & 18 \\
& 0.75 & 0.75 & 16 & 12.80 & 1.38 & 18 \\
\noalign{\smallskip}\hline\noalign{\smallskip}
\multirow{9}{*}{DSJC500\_5}
& 0.25 & 0.25 & 4 & 3.60 & 0.52 & 6 \\
& 0.25 & 0.5 & 4 & 3.90 & 0.30 & 6 \\
& 0.25 & 0.75 & 5 & 3.60 & 0.52 & 6 \\
& 0.5 & 0.25 & 6 & 5.00 & 0.63 & 8 \\
& 0.5 & 0.5 & 6 & 5.30 & 0.48 & 8 \\
& 0.5 & 0.75 & 6 & 5.20 & 0.40 & 8 \\
& 0.75 & 0.25 & 8 & 6.60 & 0.70 & 10 \\
& 0.75 & 0.5 & 8 & 6.80 & 0.63 & 10 \\
& 0.75 & 0.75 & 8 & 6.70 & 0.67 & 10 \\
\noalign{\smallskip}\hline\noalign{\smallskip}
\multirow{9}{*}{brock200\_2}
& 0.25 & 0.25 & \textbf{4} & 3.20 & 0.63 & 4 \\
& 0.25 & 0.5 & 4 & 3.20 & 0.92 & 5 \\
& 0.25 & 0.75 & 4 & 3.40 & 0.52 & 5 \\
& 0.5 & 0.25 & \textbf{6} & 4.60 & 0.92 & 6 \\
& 0.5 & 0.5 & 5 & 4.50 & 0.50 & 7 \\
& 0.5 & 0.75 & 6 & 4.90 & 0.70 & 7 \\
& 0.75 & 0.25 & 7 & 5.50 & 0.85 & 8 \\
& 0.75 & 0.5 & 7 & 5.20 & 0.63 & 8 \\
& 0.75 & 0.75 & 7 & 6.00 & 0.63 & 8 \\
\noalign{\smallskip}\hline\noalign{\smallskip}
\multirow{9}{*}{brock200\_4}
& 0.25 & 0.25 & 4 & 3.70 & 0.46 & 5 \\
& 0.25 & 0.5 & \textbf{5} & 3.90 & 0.70 & 5 \\
& 0.25 & 0.75 & \textbf{5} & 3.80 & 0.67 & 5 \\
& 0.5 & 0.25 & 6 & 5.10 & 0.70 & 8 \\
& 0.5 & 0.5 & 7 & 5.80 & 0.84 & 8 \\
& 0.5 & 0.75 & 6 & 5.20 & 0.63 & 8 \\
& 0.75 & 0.25 & 8 & 7.30 & 0.67 & 11 \\
& 0.75 & 0.5 & 8 & 6.80 & 0.63 & 11 \\
& 0.75 & 0.75 & 9 & 7.80 & 0.60 & 11 \\
\noalign{\smallskip}\hline
\end{tabular}
\end{table}

\begin{table}
\centering
\caption{Results obtained by Algorithm \ref{Alg. FW} on  instances from DIMACS dataset (part II).}
\begin{tabular}{ccccccc}
\hline\noalign{\smallskip}
\textbf{Graph} & $b$ & $p$ & \textbf{max} & \textbf{mean} & \textbf{std} & \textbf{real max}\\
\noalign{\smallskip}\hline\noalign{\smallskip}
\multirow{9}{*}{brock400\_2}
& 0.25 & 0.25 & 6 & 4.40 & 0.70 & 7 \\
& 0.25 & 0.5 & 5 & 4.10 & 0.57 & 6 \\
& 0.25 & 0.75 & 5 & 4.70 & 0.48 & 6 \\
& 0.5 & 0.25 & 7 & 6.20 & 0.70 & 9 \\
& 0.5 & 0.5 & 8 & 6.60 & 0.84 & 10 \\
& 0.5 & 0.75 & 7 & 6.30 & 0.48 & 9 \\
& 0.75 & 0.25 & 10 & 9.60 & 0.52 & 14 \\
& 0.75 & 0.5 & 11 & 9.40 & 0.92 & 15 \\
& 0.75 & 0.75 & 11 & 9.70 & 0.80 & 15 \\
\noalign{\smallskip}\hline\noalign{\smallskip}
\multirow{9}{*}{gen200\_p0.9\_44}
& 0.25 & 0.25 & 5 & 4.50 & 0.53 & 6 \\
& 0.25 & 0.5 & 5 & 4.00 & 0.67 & 6 \\
& 0.25 & 0.75 & 5 & 4.20 & 0.42 & 6 \\
& 0.5 & 0.25 & 8 & 6.70 & 0.48 & 10 \\
& 0.5 & 0.5 & 7 & 6.60 & 0.52 & 11 \\
& 0.5 & 0.75 & 7 & 6.20 & 0.75 & 10 \\
& 0.75 & 0.25 & 13 & 11.30 & 1.03 & 17 \\
& 0.75 & 0.5 & 13 & 11.60 & 0.70 & 17 \\
& 0.75 & 0.75 & 15 & 11.70 & 1.37 & 17 \\
\noalign{\smallskip}\hline\noalign{\smallskip}
\multirow{9}{*}{gen200\_p0.9\_55}
& 0.25 & 0.25 & 5 & 4.20 & 0.40 & 6 \\
& 0.25 & 0.5 & 4 & 3.80 & 0.63 & 7 \\
& 0.25 & 0.75 & 4 & 4.00 & 0.00 & 6 \\
& 0.5 & 0.25 & 8 & 7.00 & 0.63 & 10 \\
& 0.5 & 0.5 & 8 & 6.90 & 0.88 & 10 \\
& 0.5 & 0.75 & 8 & 7.00 & 0.89 & 10 \\
& 0.75 & 0.25 & 13 & 12.20 & 0.79 & 17 \\
& 0.75 & 0.5 & 15 & 12.10 & 1.52 & 17 \\
& 0.75 & 0.75 & 13 & 10.70 & 1.52 & 18 \\
\noalign{\smallskip}\hline\noalign{\smallskip}
\multirow{9}{*}{hamming8-4}
& 0.25 & 0.25 & 4 & 3.60 & 0.52 & 5 \\
& 0.25 & 0.5 & \textbf{5} & 3.80 & 0.60 & 5 \\
& 0.25 & 0.75 & 4 & 3.80 & 0.42 & 5 \\
& 0.5 & 0.25 & 6 & 4.90 & 0.57 & 7 \\
& 0.5 & 0.5 & 6 & 5.00 & 0.47 & 7 \\
& 0.5 & 0.75 & 6 & 5.20 & 0.60 & 7 \\
& 0.75 & 0.25 & 8 & 6.80 & 0.63 & 10 \\
& 0.75 & 0.5 & 8 & 6.50 & 0.62 & 10 \\
& 0.75 & 0.75 & 8 & 6.40 & 0.80 & 10 \\
\noalign{\smallskip}\hline\noalign{\smallskip}
\multirow{9}{*}{keller4}
& 0.25 & 0.25 & 4 & 3.60 & 0.52 & 5 \\
& 0.25 & 0.5 & 4 & 3.50 & 0.53 & 5 \\
& 0.25 & 0.75 & 4 & 3.20 & 0.40 & 5 \\
& 0.5 & 0.25 & 6 & 5.20 & 0.79 & 7 \\
& 0.5 & 0.5 & 6 & 5.20 & 0.63 & 7 \\
& 0.5 & 0.75 & 6 & 4.70 & 0.64 & 7 \\
& 0.75 & 0.25 & 7 & 6.60 & 0.70 & 9 \\
& 0.75 & 0.5 & 7 & 6.30 & 0.67 & 9 \\
& 0.75 & 0.75 & 7 & 6.10 & 0.30 & 9 \\
\noalign{\smallskip}\hline
\end{tabular}
\end{table}

\begin{table}
\centering
\caption{Results obtained by Algorithm \ref{Alg. FW} on instances from DIMACS dataset (part III).}
\label{table:3}
\begin{tabular}{ccccccc}
\hline\noalign{\smallskip}
\textbf{Graph} & $b$ & $p$ & \textbf{max} & \textbf{mean} & \textbf{std} & \textbf{real max}\\
\noalign{\smallskip}\hline\noalign{\smallskip}
\multirow{9}{*}{p\_hat300-1}
& 0.25 & 0.25 & \textbf{4} & 3.50 & 0.52 & 4 \\
& 0.25 & 0.5 & \textbf{4} & 3.70 & 0.46 & 4 \\
& 0.25 & 0.75 & \textbf{4} & 3.20 & 0.63 & 4 \\
& 0.5 & 0.25 & \textbf{5} & 3.80 & 0.63 & 5 \\
& 0.5 & 0.5 & 4 & 3.80 & 0.40 & 6 \\
& 0.5 & 0.75 & 4 & 3.40 & 0.52 & 5 \\
& 0.75 & 0.25 & 5 & 4.30 & 0.48 & 7 \\
& 0.75 & 0.5 & 5 & 4.40 & 0.52 & 7 \\
& 0.75 & 0.75 & 5 & 4.30 & 0.48 & 7 \\
\noalign{\smallskip}\hline\noalign{\smallskip}
\multirow{9}{*}{p\_hat300-1}
& 0.25 & 0.25 & \textbf{4} & 3.50 & 0.52 & 4 \\
& 0.25 & 0.5 & \textbf{4} & 3.70 & 0.46 & 4 \\
& 0.25 & 0.75 & \textbf{4} & 3.20 & 0.63 & 4 \\
& 0.5 & 0.25 & \textbf{5} & 3.80 & 0.63 & 5 \\
& 0.5 & 0.5 & 4 & 3.80 & 0.40 & 6 \\
& 0.5 & 0.75 & 4 & 3.40 & 0.52 & 5 \\
& 0.75 & 0.25 & 5 & 4.30 & 0.48 & 7 \\
& 0.75 & 0.5 & 5 & 4.40 & 0.52 & 7 \\
& 0.75 & 0.75 & 5 & 4.30 & 0.48 & 7 \\
\noalign{\smallskip}\hline\noalign{\smallskip}
\multirow{9}{*}{p\_hat300-2}
& 0.25 & 0.25 & 5 & 3.50 & 0.85 & 6 \\
& 0.25 & 0.5 & \textbf{5} & 4.00 & 0.47 & 5 \\
& 0.25 & 0.75 & 6 & 4.70 & 1.27 & 6 \\
& 0.5 & 0.25 & 7 & 5.40 & 0.70 & 8 \\
& 0.5 & 0.5 & 7 & 5.70 & 0.67 & 8 \\
& 0.5 & 0.75 & 7 & 5.50 & 1.00 & 9 \\
& 0.75 & 0.25 & 10 & 8.20 & 0.79 & 12 \\
& 0.75 & 0.5 & 10 & 8.80 & 0.92 & 13 \\
& 0.75 & 0.75 & 9 & 8.10 & 0.30 & 12 \\
\noalign{\smallskip}\hline\noalign{\smallskip}
\multirow{9}{*}{p\_hat300-3}
& 0.25 & 0.25 & 5 & 4.10 & 0.74 & 7 \\
& 0.25 & 0.5 & 5 & 4.70 & 0.48 & 6 \\
& 0.25 & 0.75 & 4 & 3.80 & 0.42 & 6 \\
& 0.5 & 0.25 & 7 & 6.20 & 0.63 & 9 \\
& 0.5 & 0.5 & 8 & 6.40 & 0.70 & 9 \\
& 0.5 & 0.75 & 7 & 6.00 & 0.67 & 10 \\
& 0.75 & 0.25 & 12 & 10.50 & 1.08 & 15 \\
& 0.75 & 0.5 & 11 & 10.50 & 0.50 & 16 \\
& 0.75 & 0.75 & 11 & 9.60 & 0.70 & 15 \\
\noalign{\smallskip}\hline
\end{tabular}
\end{table}

Each experiment included 10 test points and 50 matrices $\Ub$. The initial points were randomly chosen in the domain $\Delta^V_\epsilon$. The maximum number of iterations was set to $k_{\text{max}} = 1000$. The stepsize $\alpha$ was computed using the Armijo rule with parameters $\omega = 0.8$ and $\sigma = 0.4$. The stopping threshold $\xi$ was set to 0.001. For the function $G$, we used $\gamma = 1$, while the parameter $\beta$ was computed according to Theorem~\ref{Th:intequiv}, with $\epsilon = 0.001$ when $n < 1000$, and $\epsilon = 0.0001$ otherwise. Finally, we set $\delta = 0.01$ while $\eta$ was updated accordingly to Theorem~\ref{thm:final_convergence}, with $\eta_0 = \frac{\delta}{2L}$, where $L$ is the Lipschitz constant computed as in \eqref{eq:L}. We  note that with the graphs and parameters used, the algorithm  converged to a maximal common clique in all experiments.

We further highlight that the adversary oracle requires solving system \eqref{system:my_B} multiple times, which is the most computationally demanding part of the algorithm. In order to mitigate this cost, one may include a control condition that allows this step to be skipped when appropriate. In particular, if at a given iteration the algorithm reaches a feasible point $\z_k$ for which all matrices $g_\Ub$ are active, then $\UU(\z_k) = \UU_\eta(\z_k) = \UU$. In this case, the computation becomes significantly simpler: system~\eqref{system:my_B} no longer needs  to be solved, since $\partial_\eta G(\z_k)$ coincides with $\partial G(\z_k)$, and the LMO reduces to computing the Clarke dual gap $c(\z_k)$ directly.

Empirically, we have observed that this scenario arises very often, and it appears to be influenced by the choice of the problem parameters (in particular this is the case for large values of $\beta$). As a result, after a few iterations the algorithm tends to give an iterate with all $g_\Ub$ active, and keeps generating points that preserve this property. A more refined analysis of this phenomenon could provide sharper theoretical guarantees and is left for future investigation.

The results hence highlight that the proposed approach consistently delivers strong overall performance. Variability across runs remains modest, and although the algorithm falls short of the maximum common clique by a bunch of vertices, it reliably converges to large and meaningful solutions. In summary, the results confirm that the approach effectively captures the structure imposed by the backbone and shows stable behavior across different perturbation levels.

\section{Conclusion}\label{section:conclusion}

Apparently for the first time in literature, we generalize the continuous formulation of the Maximum-Clique problem to a game theoretic setting with adversarial uncertainty. Considering stable global solutions, we establish a one-to-one correspondence with maximum common cliques under this discrete uncertainty on the underlying graph structure. 
The two main tools for our continuous formulation are a penalty term and semi-continuity constraints, along with a min-max approach for dealing with the discrete uncertainty, even in presence of (exponentially) many components. 

As we are facing a difficult NP-complete problem, this is a hard global optimization task. We thus develop a first-order and projection-free algorithm, tailored to our nonsmooth and nonconvex formulation, that can be easily embedded into a global optimization algorithmic framework based on multistart or local search/basin hopping strategies.
To this end we  use a seemingly novel concept for approximate subdifferentials. While we provide a complete convergence theory, our numerical experience shows promising results.

Future lines of research may include
to establish an equivalence between local solutions to our or a related model and common maximal cliques;
from an algorithmic point of view, one can employ other variants of projection-free procedures, improving further the efficiency of the search direction determination.

\begin{acknowledgements}
 
Chiara Faccio is member of the Gruppo Nazionale Calcolo Scientifico-Istituto Nazionale di Alta Matematica (GNCS-INdAM). 
\end{acknowledgements}

%
\section*{Declarations}
The authors declare no conflict of interest.

\bibliographystyle{spmpsci}      
\bibliography{refs.bib}   

%
%

\end{document}